\documentclass[reqno,11pt]{amsart}
\usepackage{amsfonts, amsmath, amssymb, amsthm,   color}
 \usepackage{amsfonts}
\usepackage{amsmath}
\usepackage{bm}
\usepackage{enumitem}
\allowdisplaybreaks[1]

\def\bbm[#1]{\mbox{\boldmath $#1$}}

\newtheorem{thm}{Theorem}[section]
\newtheorem{lemma}[thm]{Lemma}
\newtheorem{prop}[thm]{Proposition}

\newtheorem{rmk}[thm]{Remark}

\theoremstyle{definition}
\numberwithin{equation}{section}

  \newcommand{\C}{\mathbb{C}}
\newcommand{\cal}{\mathcal }
  
  \newcommand{\N}{\mathbb{N}}

\newcommand{\rr}{\mathbb{R}}
\newcommand{\intr}{\int_{\R^2}}
\newcommand{\R}{\mathbb{R}}

\newcommand{\de}{\delta}

\newcommand{\la}{\lambda}
\newcommand{\into}{\int_{B_1}}

\newcommand{\e}{\varepsilon}
\renewcommand{\(}{\left(}
\renewcommand{\)}{\right)}

\newcommand{\beq}{\begin{equation}}
\newcommand{\eeq}{\end{equation}}

\def\bbm[#1]{\mbox{\boldmath $#1$}}
\def\Re{{\rm Re}\,}
\def\Im{{\rm Im}}

\begin{document}

\title[Higher order criteria and sharpness]{Higher-order criteria and sharpness for blow-up at quantized singularities in Liouville equations}
\author[Teresa D'Aprile]{ Teresa D'Aprile}
\author[Juncheng Wei]{Juncheng Wei}
\author[Lei Zhang]{Lei Zhang} 
\address[Teresa D'Aprile] {Dipartimento di Matematica, Universit\`a di Roma ``Tor
Vergata", via della Ricerca Scientifica 1, 00133 Roma, Italy.}
\email{daprile@mat.uniroma2.it}
\address[Juncheng Wei]{Department of Mathematics \\ Chinese University of Hong Kong \\ Shatin, NT, Hong Kong}\email{wei@math.cuhk.edu.hk}

\address[Lei Zhang]{Department of Mathematics\\
        University of Florida\\
        1400 Stadium Rd\\
        Gainesville FL 32611}
\email{leizhang@ufl.edu}

\thanks{ T. D’Aprile acknowledges the MUR Excellence
Department Project MatMod@TOV awarded to the Department of Mathematics, University of Rome
Tor Vergata, CUP E83C23000330006.  The research of J. Wei is partially supported by NSERC of Canada. The research of  Lei Zhang is partially supported by a Simons Foundation Collaboration Grant.}

\begin{abstract}
This work is devoted to the study of blow-up phenomena for solutions of $$-\Delta v= \la V(x) |x|^{2N}e^v\;\hbox{ in } B_1,\quad	v=0 \;\hbox{ on }\partial B_1,$$  where $B_1$ denotes the unit ball in $\R^2$ centered at the origin, $V(x)$ is a positive  smooth  function, $N\in\N$  and $\la>0$ is a small parameter. 
 
The main goal is to determine conditions on the potential $V$ guaranteeing the existence of families of solutions that blow up at the origin 
 as $\la\to 0^+$. 
Moreover, we establish  higher-order criteria that distinguish between simple and non-simple blow-up behavior at 
$x=0$.

In particular, for $N-1\in2^{\N}$, we construct a non-simple blow-up sequence $v_k$, corresponding to coefficient functions $V_k$ whose derivatives vanish up to the maximal order allowed by the upper bound of Wei–Zhang, thereby proving that this bound is sharp.
\bigskip


\end{abstract}

\maketitle

\section{Introduction}

Singular Liouville-type equations play an important role in several areas of nonlinear analysis and appear naturally in a variety of problems arising in conformal geometry and mathematical physics. Typical examples include the prescribed Gauss curvature equation, mean field equations, self-dual Chern–Simons models, as well as equations connected with two-dimensional turbulence and Toda systems; see, for example, \cite{calipuma, chaki, joliwa, ta1, ta2, tro, ya}.

Let $\Omega\subset\R^2$ be a smooth bounded domain containing the origin. In this paper we study the singular Liouville problem
\begin{equation}\label{eq00}
\left\{
\begin{aligned}
-\Delta v &= \la |x|^{2\gamma}V(x)e^v
&& \text{in } \Omega,\\
v&=0
&& \text{on } \partial\Omega,
\end{aligned}
\right.
\end{equation}
where $\la>0$ is a small parameter, $V$ is a smooth positive potential, and $\gamma>0$.

A major difficulty in the analysis of \eqref{eq00} stems from the exponential nonlinearity, which causes a loss of compactness in the associated variational setting. As a consequence, sequences of solutions may develop blow-up phenomena, namely concentration of mass near isolated points where the solutions become unbounded. Understanding the asymptotic structure of solutions close to these concentration points is a fundamental issue and has been extensively investigated in the literature. In particular, several refined techniques have been developed to describe the local bubbling profile near blow-up points; we refer to \cite{bachenlita, bata, breme, chenlin00, chenlin0, ChenLin, delkomu, espogropi, li, lisha, mawe, mal, nasu, su, tar, we, zha1} and the references therein.

When the blow-up point does not coincide with the singular source, the asymptotic behavior is by now well understood; see for instance \cite{breme, chenlin0, li, lisha}. A fundamental contribution in this direction is due to Li \cite{li}, who, by means of the moving-plane method, established uniform asymptotic estimates showing that solutions with a single blow-up point are well approximated by suitably rescaled standard bubbles, up to uniformly bounded errors. Later, sharper expansions were obtained by Chen--Lin \cite{ChenLin} and Zhang \cite{zhang2006, zha1}. These estimates turned out to be crucial in a number of applications, including topological degree computations.

The situation is considerably subtler when blow-up occurs at the singular point $x=0$. In this case, the singular weight strongly influences the local behavior of solutions. Results of Kuo--Lin \cite{KuoLin} and Bartolucci--Tarantello \cite{bata}, subsequently refined in \cite{bachenlita}, show that if $\gamma\notin\N$, blow-up solutions satisfy a spherical Harnack inequality near the origin. Combined with suitable Pohozaev identities, this leads to quantization properties analogous to those obtained by Li in the regular case $\gamma=0$: after rescaling, the local profile converges to an entire solution of a Liouville-type equation. More precise asymptotic estimates were later derived in \cite{zha1}.

A genuinely different phenomenon appears when the singularity strength is an integer, namely when $\gamma=N\in\N$. In this regime the origin is called a quantized singular source, and the spherical Harnack inequality may fail. As a consequence, one may observe non-simple blow-up, characterized by the formation of several local concentration points near the origin.

To describe this phenomenon more precisely, consider a sequence of solutions $v_k$ of
\begin{equation}\label{eq000}
\left\{
\begin{aligned}&
-\Delta v_k = \la_k |x|^{2N}V_k(x)e^{v_k}
&& \text{in } \Omega,\\ &
\;\;\;\;\quad v_k=0
&& \text{on } \partial\Omega,
\end{aligned}
\right.
\end{equation}
where $N\in\N$ and the functions $V_k$ are uniformly positive and bounded. Throughout the blow-up analysis one usually assumes a uniform mass bound (see \cite{batta})
\beq\label{unifbound}
\la_k\int_\Omega |x|^{2N}V_k(x)e^{v_k},dx\leq C,
\eeq
for some constant independent of $k$. We also suppose that the origin is the unique blow-up point, namely
\beq\label{eqq000}
\max_\Omega v_k\to+\infty,
\eeq
while
\beq\label{eqq0000}
v_k \text{ remains uniformly bounded on compact subsets of }
\overline\Omega\setminus{0}.
\eeq

Under these assumptions two different blow-up patterns may occur. We say that the sequence exhibits {\bf simple blow-up} at the origin if
\beq\label{simple}
\sup_{\overline\Omega}
\Big(v_k(x)+2(N+1)\log|x|+\log\la_k\Big)\leq C.
\eeq
This condition means that, after an appropriate rescaling, the solution resembles a single bubble centered at $0$
; see \cite{selfdual}.

On the contrary, we say that $v_k$ exhibits {\bf non-simple blow-up} if
\beq\label{nons}
\max_{\overline\Omega}
\Big(v_k(x)+2(N+1)\log|x|+\log\la_k\Big)\to+\infty.
\eeq
This corresponds to the breakdown of the spherical Harnack inequality. In this case, as shown in \cite{bata, KuoLin}, the solution develops $N+1$ local maxima which, after suitable scaling, become asymptotically arranged on the unit circle with equal angular separation.

The simple blow-up regime at quantized singular sources has also received considerable attention. We refer in particular to \cite{linawu} for the construction of simple blow-up solutions with non-vanishing coefficient functions, and to \cite{tuoxin} for sharp estimates for simple bubbling solutions.
However, when $N$ is an integer, non-simple blow-up may also arise. The richer blow-up behavior arising for $N\in\N$  is closely related to the structure of the limiting equation 
\beq\label{limit}
-\Delta U=|x|^{2N}e^U
\quad \text{in }\R^2,
\qquad
\int_{\R^2}|x|^{2N}e^{U(x)},dx<+\infty,
\eeq
whose solutions describe the local blow-up profile. By the classification theorem in \cite{Pratar}, all solutions are explicitly given by the three-parameter family
\beq\label{bubble}
U_{\tau,\xi}=
\log
\frac{8(N+1)^2\tau}
{\big(\tau+|x^{N+1}-\xi|^2\big)^2},
\qquad
\tau>0,\ \xi\in\C.
\eeq
In general these solutions are not radially symmetric and may possess several maxima, a feature that is responsible for the occurrence of non-simple bubbling.

Despite the extensive literature on blow-up analysis, the existence theory for non-simple blow-up solutions is still far from complete. Various constructions have been obtained under additional assumptions, for example in the presence of perturbations of the singularity, symmetry assumptions on the domain, or parameter choices depending on $\la$; see \cite{dawe, dawezhasym, delespomu}. On the other hand, the existence of non-simple blow-up solutions for fixed potentials $V$ and fixed integer singularities remains largely open, even in the unit ball, apart from a few partial results for special classes of coefficients; see \cite{dawezhaasy}.

The aim of the present paper is to contribute to this question by identifying higher-order conditions on the coefficient function that lead to simple or non-simple blow-up at the singular point.
We first focus on the case of a fixed potential $V$ in the unit ball $$B_
1
:=\big\{x\in \R^2 
: |x|<1\big\},$$
so that problem \eqref{eq000}--\eqref{eqq0000} becomes
\begin{equation}\label{eq0}
\left\{
\begin{aligned}&
-\Delta v_k
=
\la_kV(x)|x|^{2N}e^{v_k}
& \text{in }& B_1,\\&
\;\;\;\;\quad v_k=0
&\!\!\!\! \text{on } &\partial B_1,\\ &
\lambda_k\int_{B_1}V(x)|x|^{2N}e^{v_k}\,dx
\leq C.
\end{aligned}
\right.
\end{equation}
We assume that the origin is the only blow-up point, namely
\beq\label{eqq0}
\max_{B_1}v_k\to+\infty,
\qquad
\sup_{\e\leq |x|\leq1}v_k<C(\e)
\quad \forall \e>0.
\eeq
We further impose the standard assumptions
\beq\label{eqq1}
V\in C^3(\overline B_1),
\qquad
\min_{\overline B_1}V>0,
\qquad
V(0)=1,
\qquad
\nabla V(0)=0.
\eeq

We stress that the condition $\nabla V(0)=0$ is not only technical but actually necessary for blow-up at the origin. Indeed, Wei and Zhang \cite{WeiZhang1} proved that even in the absence of the spherical Harnack inequality, the gradient of the coefficient function must vanish asymptotically at the blow-up point.

Our main objective is therefore to identify additional conditions on the potential $V$, beyond \eqref{eqq1}, ensuring the occurrence of simple or non-simple bubbling at the origin for problem \eqref{eq0}--\eqref{eqq0}.

The case $N=1$ was completely analyzed in \cite{secondorder}, where a full second-order characterization of the admissible potentials was obtained in the case of nondegeneracy.

\begin{thm}[\cite{secondorder}] \label{th5bis}
Let $N=1$ and assume that $V$ satisfies \eqref{eqq1} and, in addition, that  $0$ is a nondegenerate critical point of $V$. Then a necessary and sufficient condition for the existence of a blow-up sequence $v_k$ satisfying \eqref{eq0}--\eqref{eqq0} is
$${\rm det} \,D^
2
V(0)>0.$$
Moreover, the sequence exhibits simple blow-up in the sense of \eqref{simple}, and, after suitable rescaling, its limiting profile is given by a bubble of the form \eqref{bubble}.
\end{thm}

In this context, our first result identifies higher-order  conditions on the potential $V$ ensuring blow-up at $0$ for every  $N\geq1. $
More precisely, the second-order result extends to higher orders, provided that  $i_0$ is the first index for which  $\Delta^{i_0}V(0)\neq 0$  and that all derivatives of the coefficient function vanish up to order $i_0$. 

\begin{thm}\label{th4} Let $N\geq 1$. Assume that, along with \eqref{eqq1},  
 there exists $i_0\geq 1$ such that $V$ belongs to  $C^{2i_0+1}(U)$ for a neighborhood $U$ of $0$, and 
\beq\label{mangia}\begin{aligned}&\Delta V(0)=\Delta^2 V(0)=\ldots=\Delta^{i_0-1}V(0)=0\;\;\hbox{ if }i_0\geq 2,\\ &\Delta^{i_0}V(0)\neq 0\\ &
D^\mu V(0)=0\quad \forall 1\leq |\mu|\leq i_0.\end{aligned}\eeq 
Assume moreover that
$$
2i_0\neq N+1.
$$
If $N+1<2i_0<4N+4$, assume in addition that the Taylor expansion of
$V$ at $0$, up to degree $2i_0$, has no angular Fourier components\footnote{The angular Fourier component of order $n$ of a homogeneous polynomial $P$ of degree $m$ is the component of
$
P(r,\theta)=r^m\phi(\theta)
$
associated with the Fourier frequency $n$ in the angular variable.} of
orders
$
N+1,N+2,\ldots,2N+2.
$
If $i_0=N+1$, assume further that
$
\Delta^{i_0}V(0)\neq 2\cdot 4^{i_0}(i_0!)^2.
$

Then, 
there exists a   blow-up sequence  $v_k$   satisfying \eqref{eq0}--\eqref{eqq0}.
Moreover, $v_k$ satisfies the simple blow-up property \eqref{simple} and its limiting profile, after a suitable scaling, is a radially symmetric bubble of the form \eqref{bubble} with $\xi=0$. 

\end{thm}

A more general version, with weaker restrictions on the angular Fourier components of $V$, is proved in Section 6.

In all situations described above  the solutions $v_k$ exhibit a radially symmetric limiting profile around the maximum point $0$. 
In contrast, when $2 i_0=N+1$ with $N$ odd, we are able to find  a sequence of blow-up solutions whose leading-order profile may be non-symmetric, provided that the coefficients of order $N+1$  in the expansion of $V$ at $0$ satisfy a suitable condition relating the harmonic part and the radial part, as stated in the following theorem. 

\begin{thm}\label{th5} Let $N\geq 1$ be odd. Assume that, along with \eqref{eqq1}, 
$V\in C^{N+2}(U)$ in a neighborhood $U$ of $0$ and  
\begin{itemize}
\item $\Delta V(0)=\Delta^2 V(0)=\ldots=\Delta^{\frac{N-1}{2}}V(0)=0$ if $N\geq 3$;

\vspace{2mm}

\item $\Delta^{\frac{N+1}{2}}V(0)\neq 0$;

\vspace{2mm}
\item $D^\mu V(0)=0\quad \forall 1\leq |\mu|\leq \frac{N+1}{2}$\vspace{2mm}

\item $\big|\Delta^{\frac{N+1}{2}}V(0)\big|>\kappa_N\big|{\beta}_{N+1}\big|$ with $\kappa_N= 2^{N+1}\big(\frac{N+1}{2}!\big)^2$,
\

\end{itemize}
\vspace{2mm}
where $\Re(\beta_{N+1} x^{N+1})$ denotes the harmonic component of order\footnote{\label{note}In $\R^2\simeq\C$, every homogeneous harmonic polynomial of degree $m$ is of the form
$
\Re(\overline{\beta} x^m),
\,\beta\in\C.
$
Accordingly, if $P$ is a homogeneous polynomial of degree $m$, we call
$
\Re(\overline{\beta} x^m)
$
the harmonic part of $P$ whenever
$
P(x)=\Re(\overline{\beta} x^m)+|x|^2Q(x)
$
 for some homogeneous polynomial $Q$ of degree $m-2$. Equivalently,  $\Re(\overline{\beta} x^m)$ is the projection of $P$ onto the space of harmonic homogeneous polynomials of degree $m$.} $N+1$ in the expansion of $V$ at $0$.

Then, there exists a sequence of  blow-up solutions  $v_k$   satisfying \eqref{eq0}--\eqref{eqq0}. Moreover, the limiting profile of $v_k$, after a suitable scaling, is a bubble \eqref{bubble} with $\xi$ given by  the unique solution of
\beq\label{nisc}\intr \frac{|z+\xi|}{(|1+|z|^2)^3}z \,dz +\frac\pi4 \kappa_N\frac{\beta_{N+1}}{\Delta^{\frac{N+1}{2}}V(0)}=0
.\eeq

\end{thm}

\begin{rmk}\label{rmkth5}The map $$\xi\longmapsto\intr \frac{|z+\xi|}{(|1+|z|^2)^3}z\, dz$$ is rotationally equivariant. As a consequence, its image is radial and is precisely  the open disk $|\xi|<\frac{\pi}{4}$ (see Lemma \ref{fibrafibrafibra}). It follows that equation \eqref{nisc} admits a unique solution under the stated assumptions on $\Delta^{\frac{N+1}{2}}V(0)$ and $\beta_{N+1}$. In particular, if the expansion of $V$ at $0$ has no harmonic component of order $N+1$, namely if $\beta_{N+1}=0$, then  $\xi=0$. Consequently, the corresponding sequence of blow-up solutions is radially symmetric at leading order and behaves like a single bubble centered at the origin.
\end{rmk}

In all the above results, only examples of simple blow-up are exhibited. 
We next exploit Theorem \ref{th5} to establish the existence of non-simple blow-up solutions
 for problem \eqref{eq0}–\eqref{eqq0}, by allowing the potential to depend on $k$, i.e.,  $V=V_k$.

We recall that 
in a recent  line of research developed in \cite{WeiZhang, WeiZhang1, WeiZhang2}, the authors investigate in a broad framework which vanishing properties the coefficient function $V$ must satisfy when the origin is a non-simple blow-up point. While \cite{WeiZhang1, WeiZhang2} show that in non-simple bubbling regimes both the first derivatives and the Laplacian of $V$ vanish at the quantized singular point, further results in \cite{WeiZhang4, WeiZhang3} extend these conclusions to higher-order derivatives. 

\begin{thm}[\bf{\cite{WeiZhang4}}]\label{thweizhang} Let $N\geq 2$. Assume that $v_k$ is a sequence of non-simple blow-up solutions around the origin satisfying \eqref{eq0}--\eqref{eqq0} with 
$V=V_k$ 
and
 \beq\label{assuV} \|V_k\|_{C^{N+2}(\overline B_1)}\leq C,\qquad  \frac1C\leq V_k(x)\leq C\quad \forall x\in \overline B_1\eeq for some constant $C>0$ independent of $k$.  Then,
$$D^\mu V_k(0)=o(1)\quad \forall 1\leq |\mu|\leq \max\big\{2^n+1\,\big|\, n\in\N\cup\{0\},\, 2^{n+1}\leq N\big\}.
$$
\end{thm} 

\vspace{2mm}
We now use Theorem \ref{th5} to construct non-simple blow-up sequences showing that the upper bound on the order of vanishing derivatives in Theorem \ref{thweizhang} is sharp for a distinguished class of odd values of $N$.

\begin{thm}\label{mest}  Let $N\geq1$ be odd. Then, there exists a sequence of potentials $V_k$  satisfying \eqref{assuV}  and a sequence of non-simple blow-up solutions $v_k$ for problem \eqref{eq0}--\eqref{eqq0} with $V=V_k$ such that

$$D^{\mu}V_k(0)=o(1) \hbox{ if } 1\leq |\mu|\leq \frac{N+1}{2},\qquad \max_{|\mu|=\frac{N+3}{2}}|D^{\mu} V_k(0)|\geq c>0 .$$
\end{thm}

\vspace{2mm}

\begin{rmk}
Let $N \geq 3$ be odd. Then it is immediate that
$$
\frac{N+1}{2} \geq \max\big\{2^n+1 \;\big|\; n \in \mathbb{N}\cup\{0\},\; 2^{n+1} \leq N\big\}.
$$
Moreover, equality holds if and only if $N-1 \in 2^{\N}$ (in particular, for $N=3,5,9,\ldots$).
For such values, the example of non-simple blow-up provided by Theorem \ref{mest} shows that the upper bound on the order of vanishing derivatives obtained in Theorem \ref{thweizhang} is optimal.

On the other hand, if $N-1 \notin 2^{\mathbb{N}}$, the above inequality is strict. In this case, there is a gap in the range of orders of vanishing derivatives for which the occurrence of non-simple blow-up remains an open problem.
\end{rmk}

\begin{rmk} For $N=1$   Laplacian-vanishing result  in \cite{WeiZhang2} requires $\Delta V_k(0)\to 0$  for non-simple blow-up under the usual uniform coefficient, mass, and boundary-oscillation assumptions. However, the statement in \cite{WeiZhang2} omits an additional nondegeneracy assumption, namely  $det D^2V_k(0)\geq c$, which is needed for the result to hold. Under this additional assumption, there is no contradiction with the $N=1$ construction in Theorem \ref{mest}. Indeed, the potentials $V_k$ arising in our construction do not satisfy this nondegeneracy condition, and hence the Laplacian-vanishing result does not apply to them.  \end{rmk}

The proofs are based on singular perturbation techniques. Broadly speaking, the first step is to build an approximate solution that is accurate enough to serve as a starting point for the analysis. Guided by the expected asymptotic behavior, this approximate profile is chosen to be a bubble of the form \eqref{bubble}, with a suitable parameter $\tau=\tau(\la,\xi)$. 
We then look for a true solution of \eqref{eq0}--\eqref{eqq0} in a small neighborhood of this first approximation. A key ingredient, as is standard in singular perturbation theory, is the nondegeneracy of the explicit family of solutions to the limiting Liouville equation \eqref{limit}: namely, every bounded element in the kernel of the linearized operator arises from differentiating the family with respect to its parameters, as shown in \cite{delespomu}. This property allows us to study the invertibility of the linearized operator associated with \eqref{eq0}--\eqref{eqq0} under suitable orthogonality conditions.

Next, we introduce an intermediate problem, and a fixed point argument yields a solution to an auxiliary equation, which turns out to be solvable for every choice of $\xi$. Finally, by testing the auxiliary equation against the kernel functions of the linearized operator, we obtain a reduced finite-dimensional equation. In order to construct an exact solution of \eqref{eq0}--\eqref{eqq0}, the parameter $\xi$ must therefore be a zero of this reduced map. At first order, the only zero is the trivial one, $\xi=0$, which corresponds to a radially symmetric solution. Thus, the analysis requires higher-order expansions to verify the persistence of a zero in a neighbourhood of $\xi=0$. The main technical difficulty arises in the proof of Theorem~\ref{th5}, where a more refined expansion is needed to detect a nontrivial zero $\xi$, leading to a genuinely non-symmetric solution.   

\

The paper is structured as follows. In Section 2, we fix the notation, recall some preliminary results, and construct the approximate solution. Section 3 is devoted to the analysis of the linearized problem; for complete proofs, we refer to \cite{delkomu} and \cite{espogropi}. In Section 4, we estimate the error generated by the approximation, which constitutes the most technical part of the work. Section 5 is concerned with the resolution of an auxiliary problem via a contraction mapping argument. Finally, in Section 6, we complete the proofs of Theorems  \ref{th4}, \ref{th5}, and \ref{mest}. Appendix A gathers a number of auxiliary lemmas used throughout the paper.

\

\noindent{\bf{NOTATION}}: In our estimates throughout the paper, we will frequently denote by $C>0$, $c>0$ fixed
constants, that may change from line to line, but are always
independent of the variables under consideration. We also use the
notations $O(1)$, $o(1)$, $O(\lambda)$, $o(\lambda)$ to describe
the asymptotic behaviors of quantities in a standard way.

{\Large{\bf{\section{Approximate Solution}}}}

Throughout the paper we identify $x=(x_1,x_2)\in\R^2$  with the complex variable $x=x_1+{\rm i} x_2$. 

In order to construct a sequence of bubbling solutions for the problem \eqref{eq0}-\eqref{eqq0}, we now give a construction of a suitable approximate solution.

Let us define the \textit{bubble}
$$W_\la=W_{\la, \xi}=
\log  {\la\over (\hat\la+|x^{N+1}- \xi|^2)^2},\qquad \hat\la:=\frac{\la}{8(N+1)^2} \,$$ 
which corresponds to a solution of the limit problem \eqref{limit}.

To obtain a better first approximation, we need to modify the functions $W_\la$   in order to satisfy the zero boundary condition. Precisely, we consider the projections $P W_\la $ onto the space $ H^1_0(B_1)$ of
$W_\la$, where the projection  $P:H^1(\R^2)\to  H^1_0(B_1)$ is
defined as the unique solution of the problem
$$
 \Delta P v=\Delta v\quad \hbox{in}\ B_1,\qquad  P v=0\quad \hbox{on}\ \partial B_1.
$$
Let ${\cal H}(x,y)$ denote the  regular part of the Green function in $B_1$:
\beq\label{regG}{\cal H}(x,y)=\frac{1}{2\pi}\log\Big(|x|\Big|y-\frac{x}{|x|^2}\Big|\Big)=\frac{1}{4\pi}\log\big(|x|^2|y|^2+1-2 \Re(\overline x  y)\big).
\eeq  
Let us take $\xi$ in a small neighborhood of $0$: observe that  the function ${\cal H}(x^{N+1},\xi)$ is harmonic in $B_1$ and satisfies ${\cal H}(x^{N+1},\xi)=\frac{1}{2\pi}\log |x^{N+1}-\xi|$ on $\partial B_1$. A straightforward computation gives that for any $x\in\partial B_1$ 
$$\begin{aligned}&PW_\la- W_\la+\log\la-8\pi{\cal H}(x^{N+1},\xi)\\ &=-W_\la+\log\la-4\log |x^{N+1}-\xi|
=2\log\bigg(1+\frac{\hat\la}{|x^{N+1}-\xi|^2}\bigg)
\\ & =2\frac{\hat\la}{|x^{N+1}-\xi|^2}+O(\la^2)=\frac{2\hat\la}{ 1+O(|\xi|)}+O(\la^2)=2\hat\la+O(\la|\xi|)+O(\la^2)\end{aligned}$$
uniformly for $\xi$ in a small neighborhood of $0$. 
Then the maximum principle applies and implies
the following asymptotic expansion
\beq\label{pro-exp1}\begin{aligned}
 PW_\la&= W_\la-\log\la+8\pi{\cal H}(x^{N+1},\xi)+2\hat\la+O(\la|\xi|)+O(\la^2)\\ &=-2\log\(\hat\la+|x^{N+1}- \xi|^2\)+8\pi{\cal H}(x^{N+1},\xi)+2\hat\la+O(\la|\xi|)+O(\la^2)
\end{aligned}\eeq 
uniformly for $x\in \overline B_1$ and $\xi $ in a small neighborhood of $0$.

We look for a solution to \eqref{eq0}-\eqref{eqq0} in a small neighborhood of the first approximation, namely a solution of the form
 $$v_\la=PW_\la
 + {\phi}_\la,$$ where the remainder term
$\phi_\la$ is small in
$H^1_0(B_1)$.

We conclude this section by introducing some notation and recalling a few standard facts that will be used throughout the paper. We  denote by  $\|\cdot\|$ and $\|\cdot\|_p$  the norms in  the space $ H^1_0(B_1)$ and $L^p(B_1)$, respectively, namely
 $$\|u\|:=\|u\|_{ H^1_0(B_1)}
 ,\qquad \|u\|_p:=\|u\|_{L^p(B_1)}
 \quad \forall u\in  H^1_0(B_1).$$

In the next lemma we recall the well-known Moser-Trudinger inequality
(\cite{Moe, Tru}).

\vspace{2mm}

\begin{lemma}\label{tmt} There exists $C>0$ such that 
 $$\int_{B_1} e^{\frac{4\pi u^2}{\|u\|^2}}dy\le C \quad \forall u\in{ H}^1_0(B_1).$$ 
 In particular,  for any $q\geq 1$
 $$\| e^{u}\|_{q}\le  
 C_q e^{{\frac{q}{ 16\pi}}\|u\|^2}\quad \forall  u\in{H}^1_0(B_1).$$

\end{lemma}

As mentioned in the introduction, our proof relies on  singular
perturbation methods, for which the nondegeneracy of the functions used to construct  our approximating solution is essential. The next
proposition is devoted to the nondegeneracy of the finite mass 
solutions of the singular  Liouville equation.

\begin{prop}
\label{esposito} Assume that $\phi:\R^2\to\R$  solves the problem
\begin{equation}\label{l1}
-\Delta \phi =8(N+1)^2{|y|^{2N}\over (1+|y^{N+1}-\xi|^2)^2}\phi\;\;
\hbox{in}\ \rr^2,\quad \int_{\R^2}|\nabla
\phi(y)|^2dy<+\infty.
\end{equation}
 Then there exist $c_0,\,c_1,\, c_2\in\R$
such that
$$\phi(y)=c_0  Z_0+ c_1Z_1 +c_2Z_2$$ where
$$Z_0(y):   = {1-|y^{N+1}-\xi|^2\over 1+|y^{N+1}-\xi|^2} ,\ \; \;Z_1(y):={ \Re(y^{N+1}-\xi)\over  1+|y^{N+1}-\xi|^2} 
,\ \;\;Z_2(y):={ \Im(y^{N+1}-\xi)\over  1+|y^{N+1}-\xi|^2}.
$$ \end{prop}
\begin{proof}
In \cite[Theorem 6.1]{gpistoia} it was proved that any solution
$\phi$ of \eqref{l1} is actually a bounded solution. Therefore we
can apply the result in \cite{dem} to conclude that $\phi= c_0 Z_0 + c_1
Z_1 + c_2 Z_2$ for some $c_0,c_1,c_2\in \R$.

\end{proof}

{\Large{\section{Analysis of the linearized operator}}}
According to Proposition \ref{esposito}, by the change of variable $x=\hat\la^{\frac{1}{2(N+1)}} y$, we immediately get that  all solutions $\psi $  of
$$
-\Delta \psi= {\la|x|^{2N}\over (\hat\la+|x^{N+1 }-\xi|^2)^2}\psi =|x|^{2N} e^{W_\la}\psi\quad \hbox{in}\;\; \rr^2,\qquad \intr |\nabla \psi|^2 dx<+\infty$$
are linear combinations of the functions
$$Z^0_{\la}(x)={\hat\la-|x^{N+1}-\xi|^2\over \hat\la+|x^{N+1 }-\xi|^2},\ Z^1_{\la}(x)=
{{\hat\la}^\frac12\Re(x^{N+1 }-\xi)\over  \hat\la+|x^{N+1 }-\xi|^2},\ Z^2_{\la}(x)=
{ {\hat\la}^\frac12 \Im(x^{N+1 }-\xi)\over  \hat\la+|x^{N+1 }-\xi|^2}
.$$
We introduce their projections $PZ^j_{\la}$ onto $H^1_0(B_1).$ It is immediate that
 $$
PZ^0_{\la}(x)=Z^0_{\la}(x)+1+ O(\la)
$$
 and
  \begin{equation}\label{pzi}
PZ^j_{\la}(x)=Z^j _{\la}(x) + O\big({\la}^\frac12\big),\;\; j=1,2\end{equation}
uniformly with respect to $x\in\overline B_1$ and $\xi$ in a small neighborhood of $0$.
\\
We need to expand  the projections $PZ_\la^j$ to a higher order with respect to \eqref{pzi}.

\vspace{2mm}

\begin{lemma}
\label{pziraff} The following holds 
$$\begin{aligned}&PZ^1_\la(x)=Z_\la^1 (x)-\hat\la^\frac12\Re (x^{N+1})+O(\la^\frac12 |\xi|)+O(\la^\frac32\big),\\ &PZ^2_\la(x)=Z_\la^2 (x)-\hat\la^\frac12\Im (x^{N+1})+O(\la^\frac12 |\xi|)+O(\la^\frac32\big)\end{aligned}$$
uniformly for $\xi$ in a small neighborhood of $0$.
\end{lemma}
\begin{proof} Observe that 
$$\begin{aligned}\hbox{if } |x|=1: \;\;Z_\la^1(x)&=\frac{\hat\la^\frac12\Re(x^{N+1}-\xi)}{\hat\la+|x^{N+1}-\xi|^2}=\frac{\hat\la^\frac12\Re(x^{N+1}-\xi)}{1+\hat\la+O(|\xi|)
}\\ &=\hat\la^\frac12\Re(x^{N+1}-\xi)\big(1+O(\la)+O(|\xi|)\big)
\\ &=\hat\la^\frac12\Re(x^{N+1}) +O(\la^\frac32)+O(\la^\frac12|\xi|)
.\end{aligned}$$
Therefore, if we set  $$ \hat{Z}_\la^1:= Z_\la^1 (x)-\hat\la^\frac12\Re(x^{N+1})$$ we get $$\hat Z_\la^1=O(\la^\frac32)+O(\la^\frac12|\xi|)
 \hbox{ if }|x|=1$$ and 
$$-\Delta \hat Z_\la^1 (x)=-\Delta Z_\la^1(x)=-\Delta PZ_\la^1(x) \hbox{ in }B_1.$$
Hence, since by construction $PZ_\la^1=0$ for $|x|=1$, the maximum principle applies and gives 
$$PZ_\la^1= \hat Z_\la^1+O(\la^\frac32)+O(\la^\frac12 |\xi|) \hbox{ in } B_1.$$
\end{proof}

From now on, let $M>0$ be a sufficiently large number, to be chosen later.

We perform the following estimates.

\vspace{2mm}

\begin{lemma}\label{aux00}   Let
 $0\leq m<2N+2$ and $p\geq 1$ be fixed. The following holds:\beq\label{aux00esti1}
\|  |x|^{2N+m}e^{W_{\la}}\|_p
\leq C\la^{\frac{m}{2(N+1)}}\la^{-\frac{p-1}{(N+1)p}} ,\quad \|  \la |x|^{2N+m} e^{PW_{\la}}\|_p
\leq C\la^{\frac{m}{2(N+1)}}\la^{-\frac{p-1}{(N+1)p}} \eeq
uniformly for $|\xi|\leq M\hat\la^\frac12$. 
\end{lemma}
\begin{proof} 
We compute $$\begin{aligned}\||x|^{2N+m} e^{W_{\la}}\|_p^p&=
\la^{p}\into\frac{|x|^{(2N+m )p}}{(\hat\la+|x^{N+1}-\xi|^2)^{2p}}dx
\leq C\la^{ m \frac{p}{2(N+1)}-\frac{p-1}{N+1}} \intr\frac{|z|^{(2N+m )p}}{(1+|z|^{2(N+1)})^{2p}}dz
\end{aligned}$$ uniformly for $|\xi|\leq M\hat\la^\frac12$.
Taking into account that the last integral is finite for   $0\leq m<2N+2$ and $p\geq 1$ 
we deduce the first part of  \eqref{aux00esti1}.
To prove the second part it is sufficient to observe that 
by \eqref{pro-exp1} 
we derive
$$\la e^{PW_{\la}}= e^{W_{\la}+O(1)}
=O(e^{W_{\la}}).$$
\end{proof}

Let us consider the following linear problem: given
$ h\in H_0^1(B_1)$,  find a function $\phi\in  H^1_0(B_1)$ and constants $c_1,c_2\in\R$ satisfying
\begin{equation}\label{lla2}
\left\{\begin{aligned}&-\Delta \phi   -\la V(x)|x|^{2N} e^{P{W}_\la}\phi=\Delta h+\sum_{j=1,2}c_j Z^j_{\la} |x|^{2N}e^{W_{\la}}\\ &\into \nabla \phi\nabla PZ_\la^j=0\;\;j=1,2\end{aligned}\right..
\end{equation}

In order to solve problem \eqref{lla2}, we need to establish an a priori estimate. For the proof we refer to \cite{delkomu} (Proposition 3.1) or \cite{espogropi} (Proposition 3.1) for the corresponding result in the case of regular bubbles; the proof carries over, with only minor modifications, to the singular bubbles considered here.

\vspace{2mm}

\begin{prop}\label{inv} 
There exist $\lambda_0>0$ and $C>0$ such that for any $\la \in(0, \la_0)$, any $\xi\in \R^2$ with $|\xi|<M\hat\la^\frac12$ and   any $h\in  H^1_0(B_1)$, if  
$(\phi,c_1,c_2)\in  H^1_0(B_1)\times \R^2$  solves  \eqref{lla2}, then the following holds $$\|\phi\|    \leq C  |\log\la |    \|h\| .$$
\end{prop}

\

Next, for any $ p>1,$ let $$i^*_{p}:L^{p}(B_1)\to H^1_0(B_1)$$ be the
adjoint operator of the embedding
$i_{p}:H^1_0(B_1)\hookrightarrow L^{p\over p-1 }(B_1),$ i.e.
$u=i^*_{p}(v)$ if and only if $-\Delta u=v$ in $B_1,$ $u=0$ on
$\partial B_1.$ We point out that $i^*_{p}$ is a continuous
mapping, namely
\begin{equation}
\label{isp} \|i^*_{p}(v)\| \le c_{p} \|v\|_{p}, \ \hbox{for any} \ v\in L^{p}(B_1),
\end{equation}
for some constant $c_{p}$ which depends on $p.$
Next let us set
$$      {K }:=\hbox{span}\left\{PZ^1_{\la},\ PZ^2_{\la}\right\}$$
and
$$      {K ^\perp }:= \left\{\phi\in H^1_0(B_1)\ :\ \into \nabla \phi \nabla PZ^1_{\la}dx= \into \nabla \phi \nabla PZ^2_{\la}dx= 0\right\} $$
and  denote by
$$     \Pi : H^1_0(B_1)\to       {K },\qquad      {\Pi ^\perp}: H^1_0(B_1)\to       {K ^\perp }$$
the corresponding projections.
Let $      L: K^\perp \to K^\perp$  be the
linear operator defined by
$$
      L(     \phi):=  \Pi^{\perp}\Big(   {i^*_{p}}\big(       \la V(x)|x|^{2N}e^{PW_\la}     \phi \big)\Big) - \phi. 
$$
Notice that problem \eqref{lla2} reduces to $$L(\phi)=\Pi^\perp h, \quad \phi\in K^\perp.$$
 
 As a consequence of Proposition \ref{inv} we derive the invertibility of $L$.

\begin{prop}\label{ex} For any $p>1$ there exist $\lambda_0>0$ and $C>0$ such that for any $\la \in(0, \la_0)$, any $\xi\in\R^2$ with $|\xi|<M\hat\la^\frac12$ and 
 any $h\in K^\perp$  there is a unique solution $ \phi\in K^\perp$ to the problem $$L(\phi)=h.$$ In
particular, $L$ is invertible; moreover, $$\| L^{-1} \| \leq C
|\log \lambda |.$$

\end{prop}
\begin{proof}  Observe that the operator $\phi\mapsto \Pi^\perp\big( {i^*_{p}}\(       \la V(x)|x|^{2N}e^{PW_\la}     \phi \)\big)$ is a compact operator in $K^\perp$.
   Let us consider the case $h=0$, and take $\phi\in K^\perp$ with $L(\phi)=0$. In other words,  $\phi$ solves
the system \eqref{lla2} with $h=0$ for some $c_1,c_2\in\R$. Proposition \ref{inv} implies $\phi\equiv 0$. Then, Fredholm's alternative implies the existence and uniqueness result.

Once we have existence, the norm estimate follows directly from Proposition \ref{inv}.
\end{proof}

{\Large{\section{Estimate of the error term}}}

Assume that, along with \eqref{eqq1},  there exists $i_0\geq1$ such that  the potential $V$ is of class $C^{2i_0+1}(U)$ in a neighborhood $U$ of $0$ and  \eqref{mangia} holds.  Then, we perform the following expansion:
$$V(x)=1+ {\rm q}_{i_0+1}(x)+  {\rm q}_{i_0+2}(x)+\ldots +{\rm q}_{2i_0}(x)+O(|x|^{2i_0+1}),
$$ where $   {\rm q}_i(x)$ denotes the Taylor polynomial homogeneous of degree $i$.

In particular, let us write explicitly $  {\rm q}_i(x)$ whose coefficients will play a crucial role: if we write each $  {\rm q}_i(x)$ in polar coordinates $(r,\theta)$, so that on the unit circle $x=e^{{\rm i}\theta}$ we have that $${\rm q}_{i}(e^{{\rm i}\theta}) = \alpha_{i}+\Re\big(\overline\beta_ie^{{\rm i}i \theta}\big) +\sum_{\ell=1}^{i-1}\Re\big(\overline\gamma_{i\ell}e^{{\rm i}\ell \theta}\big)
,\qquad \theta\in \R,\;i\geq 2,$$ 
where $\alpha_i=\frac{1}{2\pi}\int_{\partial B_1}   {\rm q}_i(x) dx\in\R$ and $\beta_i,\,\gamma_{i\ell}\in\C$. 
Moreover, since $q_i$ is an even (resp. odd) homogeneous polynomial when $i$ is even (resp. odd), only even (resp. odd) frequencies appear. So, $  {\rm q}_i$ can be rewritten as 
\beq\label{homo1}{\rm q}_{i}(x)=\alpha_{i}|x|^{i}+\Re\big(\overline\beta_i x^i\big)
+|x|^{i}\sum_{\ell=1\atop i+\ell \hbox{ {\tiny{even}}}}^{i-1}\frac{\Re(\overline\gamma_{i\ell} x^{\ell})}{|x|^{\ell}}
,\eeq
with $\alpha_i=0$ for $i$ odd. 
We stress that the quantity $\Re\big(\overline \beta_i x^i)$ coincides with the harmonic part of the homogeneous polynomial ${\rm q}_i$.

The next lemma expresses the coefficient $\alpha_{2i}$ in terms of the $i$-Laplacian of $V$ at $0$.

\

\begin{lemma}\label{finallemma} 
The following holds: 
$$\Delta^{i}{\rm q}_{2i}(x)= 4^i(i!)^2 \alpha_{2i},\quad i\geq 1.$$
Consequently,
$$\Delta^iV(0)=\Delta^i{\rm q}_{2i}(0)=4^i(i!)^2 \alpha_{2i}.$$
\end{lemma}

\begin{proof}  
Let us fix  $h=1,\ldots, i-1$; then 
$$\begin{aligned}&\Delta(|x|^{2(i-h)}\Re(x^{2h}))=\Delta(|x|^{2(i-h)})\Re(x^{2h})+2\nabla (|x|^{2(i-h)})\cdot\nabla (\Re(x^{2h}))\\ &=
4(i-h)^2|x|^{2(i-h-1)}\Re(x^{2h})+4(i-h)|x|^{2(i-h-1)}2h \big(x_1\Re(x^{2h-1})-x_2\Im(x^{2h-1})\big)\\ &=4(i-h)(i+h)|x|^{2(i-h-1)}\Re(x^{2h})
\end{aligned}$$ and, by iterating we arrive at 
$$\begin{aligned}&\Delta^{i-h}(|x|^{2(i-h)}\Re(x^{2h}))=c_{ih}\Re(x^{2h})\end{aligned}$$ for some constants $c_{ih}\in\R$, 
by which, since $\Re(x^{2h})$ is harmonic,  $$\Delta^i(|x|^{2(i-h)}\Re(x^{2h}))=c_{ih}\Delta^h(\Re(x^{2h}))=0.$$
Similarly
$$\Delta^i(|x|^{2(i-h)}\Im(x^{2h}))=c_{ih}\Delta^h(\Im(x^{2h}))=0.$$
So, the definition of ${\rm q}_{2i}$ in \eqref{homo1} gives

$$\Delta^i({\rm q}_{2i}(x))=\alpha_{2i}\Delta^i(|x|^{2i})=4^i(i!)^2 \alpha_{2i}.$$
\end{proof}

The main goal of this section is to provide an estimate of the error up to which the  approximate solution $PW_\la$ solves problem \eqref{eq0}.

For $N\geq 1$ define 
\beq\label{errelambda}{R}_\la:=-\Delta PW_\la-\la V(x) |x|^{2N}e^{PW_\la}=|x|^{2N}e^{W_\la}-\la V(x) |x|^{2N}e^{PW_\la}.
\eeq

We begin by establishing the following intermediate error estimate.

\begin{lemma}\label{semierror} Let $N\geq 1$ and assume that  $V$ satisfies \eqref{eqq1}. Then, the following holds 
$$\begin{aligned}R_\la&=-2\hat\la +4\Re(\bar\xi x^{N+1})-\big(V(x)-1\big)|x|^{2N}e^{W_\la}\\ &\;\;\;\; 
+\big(V(x)-1\big)\big(|\la|+|\xi| |x|^{N+1}\big)O\big(|x|^{2N}e^{W_\la}\big)+\big(|\xi|^2|x|^{2(N+1)}+\la|\xi|+\la^2\big)O\big(|x|^{2N}e^{W_\la}\big)\end{aligned}$$
uniformly for $\xi$ in a small neighborhood of $0$. 
As a corollary 
\beq\label{litilitio}R_\la=\big(\la+|\xi||x|^{N+1}+|V(x)-1|\big)O\big(|x|^{2N}e^{W_\la}\big)\eeq 
uniformly for $\xi$ in a small neighborhood of $0$. 

\end{lemma}

\begin{proof} By \eqref{regG} and  \eqref{pro-exp1}  
we derive $$\begin{aligned}&\la |x|^{2N} e^{PW_\la} \\&
=|x|^{2N} e^{W_\la +8\pi {\cal H}(x^{N+1},\xi)+2\hat\la+O(\la|\xi| )+O(\la^2)} 
\\ &=|x|^{2N}e^{W_\la}(1+|x|^{2(N+1)}|\xi|^2-2\Re( \overline \xi x^{N+1}))^2+|x|^{2N}e^{W_\la}(2\hat\la+O(\la|\xi|)+O(\la^2))
\\ &=|x|^{2N}e^{W_\la}\big(1+ 2\hat\la-4\Re(\bar\xi x^{N+1})\big)+O\big(|x|^{2N}e^{W_\la}\big)\big(|\xi|^2|x|^{2(N+1)}+\la|\xi|+\la^2\big)
\end{aligned}$$
uniformly for $\xi$ in a small neighborhood of $0$. The definition of the error $R_\la$ in \eqref{errelambda} allows us to conclude.

\end{proof}

Let us introduce the following integral quantities which will play a key role for the error estimates:
\beq\label{gammai}\Gamma_{2i}:= 8i\intr|y|^{\frac{2i}{N+1}}\frac{1}{(1+|y|^{2})^3}dy ,\qquad  0\leq\frac{ i}{N+1}<2 .\eeq

In the following, writing $$Z_\la=(Z_\la^1, Z_\la^2),\qquad PZ_\la=(PZ_\la^1,PZ_\la^2)$$ we interpret the integrals componentwise as vectors in $\mathbb{R}^2$. 

\begin{lemma}\label{santa00} Let  $0\leq i<\frac32(N+1).$ Then,  $$\intr |x|^{2(N+i)}e^{W_\la} | Z_\la|dx=O(\la^{\frac{i}{N+1}}),\quad\intr |x|^{2(N+i)}e^{W_\la}  Z_\la dx=  \Gamma_{2i}\hat\la^{\frac{i}{N+1}-\frac12}\xi +O\big(\hat\la^{\frac{i}{N+1}-1}|\xi|^2\big)$$ uniformly for $|\xi|\leq M {\hat\la}^\frac12$.
\end{lemma}
\begin{proof}
By applying first Lemma \ref{copy} and next the change of variable $y=\hat\la^\frac12 z$, 
$$\begin{aligned} \intr |x|^{2(N+i)}e^{W_\la}  Z_\la dx&=\frac{\la \hat\la^{\frac12}}{N+1}\intr|y|^{\frac{2i}{N+1}}\frac{y-\xi}{(\hat\la+|y-\xi|^{2})^3}dy\\ &=8(N+1)\hat\la^{\frac{i}{N+1}}\intr|z|^{\frac{2i}{N+1}}\frac{z-\hat\la^{-\frac12}\xi}{(1+|z-\hat\la^{-\frac12}\xi|^{2})^3}dz\\ &=8(N+1)\hat\la^{\frac{i}{N+1}}f(\hat\la^{-\frac12}\xi)\\ &=  \Gamma_{2i}\hat\la^{\frac{i}{N+1}-\frac12}\xi+O\big(\hat\la^{\frac{i}{N+1}-1}|\xi|^2\big)
\end{aligned}$$ uniformly for $|\xi|\leq M  \hat\la^{\frac12}$, 
where $f=(f_1,f_2)$ is the vector field defined in Lemma \ref{fibra}. 
Moreover, arguing similarly,
$$\begin{aligned} \intr |x|^{2(N+i)}e^{W_\la}  |Z_\la|dx&=8(N+1)\hat\la^{\frac{i}{N+1}}\intr|z|^{\frac{2i}{N+1}}\frac{|z-\hat\la^{-\frac12}\xi|}{(1+|z-\hat\la^{-\frac12}\xi|^{2})^3}dz=O(\la^{\frac{i}{N+1}})\end{aligned}$$ uniformly for $|\xi|\leq M  \hat\la^{\frac12}$.
\end{proof}
In the next  proposition we quantify the error bound for the  approximation in different settings.

\begin{prop}\label{auxaux} Let $N\geq 1$. Assume  that, along with \eqref{eqq1},   there exists $i_0\in\N$ with $1\leq i_0\leq N+1$ such that $V\in C^{2i_0+1}(U)$ in a neighborhood of $0$  and \eqref{mangia} holds.  

 Then,  for any fixed $p\geq 1$ the following holds:
\beq\label{cese11}\|R_\la\|_{p}\leq C\la^{\frac{i_0+1}{2(N+1)}}
\la^{-\frac{p-1}{(N+1)p}} \eeq uniformly for $|\xi|\leq M  \hat\la^{\frac12}$.
Moreover: 

\begin{itemize}[leftmargin=*]
\item if $N\geq 2i_0$ one has

 \beq\label{dublend1}\into R_\la PZ_\la dx
 + \alpha_{2i_0}\Gamma_{2i_0}\hat\la^{\frac{i_0}{N+1}-\frac12}\xi
 =O\Big(\la^{\frac{2i_0+1}{2(N+1)}}
 +|\xi|^2\la^{\frac{i_0}{N+1}-1}\Big)
 \eeq
 uniformly for $|\xi|\leq M  \hat\la^{\frac12}$, where the numbers $\Gamma_{2i}$ are defined in \eqref{gammai};

\
 
 \item if $2i_0=N+1$ one has \beq\label{dublend3}\into R_\la PZ_\la dx+ 8(N+1)\alpha_{N+1}\hat\la^\frac12{\cal F}(\hat\la^{-\frac12} \xi) +2(N+1)\pi \hat\la^\frac12\beta_{N+1}=O\big(\la^{\frac12+\frac{1}{2(N+1)}}
 \big)
 \eeq
uniformly for $|\xi|\leq M  \hat\la^{\frac12}$, where ${\cal F}=({\cal F}_1,{\cal F}_2)$ is the vector field defined in Lemma \ref{fibrafibrafibra};

\

 \item if $N+1<2i_0<2(N+1)$, the expansion \eqref{dublend1} has  additional terms
  \beq\label{dublend2}\begin{aligned}&\into R_\la PZ_\la dx+  \alpha_{2i_0}\Gamma_{2i_0}\hat\la^{\frac{i_0}{N+1}-\frac12}\xi +2\pi(N+1)\hat\la^\frac12
  \beta_{N+1}\\ &=O\big(\la^{\frac{2i_0+1}{2(N+1)}}+|\xi|^2\la^{\frac{i_0}{N+1}-1}\big)+\sum_{i=N+2}^{2i_0} \gamma_{i,N+1}O\big(\la^{\frac{i}{2(N+1)}}\big)\end{aligned}
 \eeq
uniformly for $|\xi|\leq M \hat\la^\frac12$;

\

\item if $i_0=N+1$, one has 
\beq\label{dublend4}\begin{aligned}&\into R_\la PZ_\la dx+4\pi (N+1)\hat\la^\frac12\big( \alpha_{2N+2}\xi+\beta_{2N+2}\overline \xi-2\xi \big)
\\ &=O\big(\la^{1+\frac{1}{2(N+1)}}\big)+\sum_{i=N+2}^{2N+2} \gamma_{i,N+1}O\big(\la^{\frac{i}{2(N+1)}}\big)\end{aligned}
 \eeq
uniformly for $|\xi|\leq M \hat\la^\frac12$.

  \

\end{itemize}
 \end{prop}
\begin{proof} Taking into account that  $V(x)-1=\sum_{i=i_0+1}^{2i_0}   {\rm q}_i(x) +O(|x|^{2i_0+1})=
 O(|x|^{i_0+1})$  according to \eqref{mangia}, by \eqref{litilitio} we get
 
 \beq\label{imbe}R_\la=(\la+|\xi||x|^{N+1}+|x|^{i_0+1}\big)O(|x|^{2N}e^{W_\la})\eeq
 and  \eqref{cese11} can be deduced directly by Lemma \ref{aux00}. 
Next,  using Lemma \ref{semierror}  we also get
\beq\label{depi}\begin{aligned}&R_\la +\big(2\hat\la- 4\Re(\bar\xi x^{N+1})\big) |x|^{2N}e^{W_\la}+ \sum_{i=i_0+1}^{2i_0}   {\rm q}_i(x)  |x|^{2N}e^{W_\la}
\\ &=\big(|\xi|^2|x|^{2(N+1)}+
\la|\xi|+\la^2+|x|^{2i_0+1}
+\la|x|^{i_0+1}+|\xi||x|^{N+i_0+2}
\big)O\big(|x|^{2N}e^{W_\la}\big) \end{aligned}\eeq uniformly for $|\xi|\leq M  \hat\la^{\frac12}$.
By Lemma \ref{copy0} we deduce
$$\intr |x|^{2N}e^{W_\la}Z_\la dx=0$$ and, in view of Remark \ref{enne=4} we have 
$$\intr \Re(\bar\xi x^{N+1}) |x|^{2N}e^{W_\la}Z_\la dx=2\pi(N+1)\hat\la^\frac12\xi.$$
 Therefore,  by  testing the above identity \eqref{depi} against $ Z_\la$ and using  Lemma \ref{santa00},
\beq\label{igo}\begin{aligned} &\into R_\la Z_\la dx-8\pi(N+1)\hat\la^\frac12\xi
+ \sum_{i=i_0+1}^{2i_0}\intr|x|^{2N}   {\rm q}_i(x)e^{W_\la }  Z_\la dx=O\big(\la^{\frac{2i_0+1}{2(N+1)}}\big)\end{aligned}\eeq uniformly for $|\xi|\leq M  \hat\la^{\frac12}$.
According to Lemma \ref{pziraff}  and \eqref{imbe} we derive
$$R_\la(Z_\la-PZ_\la)=(\la+|x|^2)\big(\la^\frac12|x|^{N+1}+\la^\frac12|\xi|+\la^\frac32\big)
O\big( |x|^{2N}e^{W_\la}\big)$$ uniformly for $|\xi|\leq M  \hat\la^{\frac12}$, which implies, using Lemma \ref{aux00}, 
\beq\label{uffa}\into R_\la PZ_\la dx=
 \into R_\la Z_\la dx+ O\big(\la^{1+\frac{1}{2(N+1)}}
\big)\eeq uniformly for $|\xi|\leq M  \hat\la^{\frac12}$.

\

Assume for the moment that $$N\geq2 i_0.$$
Now we use  the definition of $ {\rm q}_i$ in \eqref{homo1}: taking into account that  $|x|^{2N}e^{W_\la } Z_\la$ is $\frac{2\pi}{N+1}$ symmetric, then, if $N\geq 2i_0$ we apply Lemma \ref{copy0} with $(\ell, m)= (i, N+1)$    for $i=i_0+1,\ldots, 2i_0$ and obtain immediately  
\beq\label{santa3}\sum_{i=i_0+1}^{2i_0}\intr|x|^{2N}  {\rm q}_i(x)e^{W_\la }  Z_\la dx=\!\!\sum_{i=i_0+1\atop i\hbox{ {\tiny{even}}} }^{2i_0}\intr \alpha_i|x|^{2N+i}  e^{W_\la }  Z_\la dx
= \alpha_{2i_0}\!\intr|x|^{2(N+i_0)} e^{W_\la }  Z_\la dx,\eeq
where last identity follows by $\alpha_i=0$ for $i<2i_0$ according to  \eqref{mangia} and Lemma \ref{finallemma}. 
By inserting \eqref{uffa}--\eqref{santa3}  into \eqref{igo}  we arrive at 
$$\begin{aligned} \into R_\la P Z_\la  dx
+\alpha_{2i_0}\intr|x|^{2(N+i_0)} e^{W_\la }  Z_\la dx 
= 
O\big(\la^{\frac{2i_0+1}{2(N+1)}}\big)
\end{aligned}$$
 uniformly for $|\xi|\leq M\hat\la^\frac12$. 
 By using Lemma \ref{santa00} we deduce \eqref{dublend1}  
 and the thesis  is then established for $N\geq 2i_0$. 

\

 Similarly, if $N+1< 2i_0<2(N+1)$,  using again the definition of $  {\rm q}_i$ in \eqref{homo1} and by applying Lemma \ref{copy0} with $\ell< N+1$ and $N+1<\ell<2(N+1)$,
 \beq\label{santa33}\begin{aligned}&\sum_{i=i_0+1}^{2i_0}\intr|x|^{2N} {\rm q}_i(x)e^{W_\la }  Z_\la dx \\ &= \alpha_{2i_0}\intr|x|^{2(N+i_0)} e^{W_\la }   Z_\la dx+
\sum_{i=N+2}^{2i_0}\intr \Re(\overline\gamma_{i,N+1} x^{N+1})|x|^{N+i-1} e^{W_\la }   Z_\la dx\\ &\;\;\;\;
+ \intr \Re\big(\overline \beta_{N+1}x^{N+1}\big)|x|^{2N} e^{W_\la }   Z_\la dx
 \\ &
= \alpha_{2i_0}\intr|x|^{2(N+i_0)} e^{W_\la } Z_\la dx+2\pi (N+1) \hat\la^{\frac12}\beta_{N+1}+\sum_{i=N+2}^{2i_0} \gamma_{i,N+1}O\big(\la^{\frac{i}{2(N+1)}}\big)
  \end{aligned}\eeq according to Remark  \ref{enne=4} and Lemma \ref{santa00}. 
Combining  \eqref{uffa} and \eqref{santa33}  with \eqref{igo} we arrive at
\beq\label{odi}\begin{aligned}&\into R_\la PZ_\la dx+ \alpha_{2i_0}\intr|x|^{2(N+i_0)} e^{W_\la } Z_\la dx +2\pi (N+1)\hat\la^\frac12 \beta_{N+1}\\ &=O\big(\la^{\frac{2i_0+1}{2(N+1)}}\big) +\sum_{i=N+2}^{2i_0} \gamma_{i,N+1}O\big(\la^{\frac{i}{2(N+1)}}\big)
.\end{aligned}\eeq
If $N+1< 2i_0< 2(N+1)$, then \eqref{dublend2} can be deduced  directly by \eqref{odi} applying  Lemma \ref{santa00}. 

\

On the other hand, if $2i_0=N+1$, by applying Lemma \ref{copy0} with $\ell< N+1$,
 \beq\label{santa33tris}\begin{aligned}&\sum_{i=i_0+1}^{2i_0}\intr|x|^{2N} {\rm q}_i(x)e^{W_\la }  Z_\la dx  = \sum_{i=\frac{N+3}{2}}^{N+1}\intr|x|^{2N} {\rm q}_i(x)e^{W_\la }  Z_\la dx\\ &=\alpha_{N+1}\intr|x|^{3N+1} e^{W_\la }   Z_\la dx+ \intr \Re\big(\overline \beta_{N+1}x^{N+1}\big)|x|^{2N} e^{W_\la }   Z_\la dx
 \\ &
= \alpha_{N+1}\intr|x|^{3N+1} e^{W_\la } Z_\la dx+2\pi (N+1) \hat\la^{\frac12}\beta_{N+1}
  \end{aligned}\eeq according to Remark  \ref{enne=4}; 
by applying first Lemma \ref{copy} and next the change of variable $y=\hat\la^\frac12 z$, 
$$\begin{aligned} 
 \intr |x|^{3N+1}e^{W_\la}  Z_\la dx &=\frac{\la \hat\la^{\frac12}}{N+1}\intr|y|\frac{y-\xi}{(\hat\la+|y-\xi|^{2})^3}dy \\&=8(N+1) \hat\la^{\frac12}\intr|z|\frac{z-\hat\la^{-\frac12}\xi}{(1+|z-\hat\la^{-\frac12}\xi|^{2})^3}dz\\ &=8(N+1)\hat\la^\frac12{\cal F}(\hat\la^{-\frac12} \xi)
\end{aligned}$$
where ${\cal F}=({\cal F}_1,{\cal F}_2)$ is the vector field defined in Lemma \ref{fibrafibrafibra}. By substituting in \eqref{santa33tris} and next into \eqref{igo} (using also \eqref{uffa}), we obtain \eqref{dublend3} when $2i_0=N+1$.

\vspace{2mm}

Finally, assume  $i_0=N+1$; then 
by applying Lemma \ref{copy0}  for  $N+2\leq \ell<2(N+1)$ 
 \beq\label{santa33tre}\begin{aligned}&\sum_{i=i_0+1}^{2i_0}\into|x|^{2N} {\rm q}_i(x)e^{W_\la }  Z_\la dx=\sum_{i=N+2}^{2N+2}\into|x|^{2N+2} {\rm q}_i(x)e^{W_\la }  Z_\la dx
 \\ &= \alpha_{2N+2}\into|x|^{4N+2} e^{W_\la } Z_\la+\intr \Re\big(\overline \beta_{2N+2}x^{2(N+1)}\big)
 |x|^{2N} e^{W_\la }   Z_\la dx \\ &\;\;\;\;+\sum_{i=N+2}^{2N+2}\intr\Re(\overline\gamma_{i,N+1} x^{N+1})
 |x|^{N+i-1} 
 e^{W_\la }   Z_\la dx
\\ &=
 \alpha_{2N+2}\into|x|^{4N+2} e^{W_\la } Z_\la+4\pi(N+1)\hat\la^\frac12 \beta_{2N+2}\overline\xi+\sum_{i=N+2}^{2N+2} \gamma_{i,N+1}O\big(\la^{\frac{i}{2(N+1)}}\big) \end{aligned}\eeq uniformly for $|\xi|\leq M\hat\la^\frac12$, according to Remark  \ref{enne=4} and Lemma \ref{santa00}.  On the other hand$$\begin{aligned}
 \intr |x|^{4N+2}e^{W_\la}  Z_\la dx &=
\frac{\la \hat\la^{\frac12}}{N+1}\intr|y|^2\frac{y-\xi}{(\hat\la+|y-\xi|^{2})^3}dy\\ &=16(N+1)\hat\la^\frac12 \intr\Re(\overline \xi z)\frac{z}{(1+|z|^{2})^3}dz\\ &=4\pi(N+1)\hat\la^\frac12 \xi
\end{aligned}$$
since $\intr \frac{z_1^2}{(1+|z|^2)^3}dz=\frac\pi4$. By substituting in \eqref{santa33tre}, and next into \eqref{igo},  
(using also \eqref{uffa})
we deduce \eqref{dublend4} for $i_0=N+1$. 
 \end{proof}

\begin{rmk}\label{enne=4} 

By applying first Lemma \ref{copy} and next the change of variable $y=\hat\la^{1/2} z$ we can evaluate 
$$\begin{aligned} \intr \Re(x^{N+1})|x|^{2N}e^{W_\la} Z_\la&=
\frac{\la}{N+1} \hat\la^{\frac12}\intr\frac{y_1(y-\xi)}{(\hat\la+|y-\xi|^2)^3} dy
=\frac{\la}{N+1} \hat\la^{\frac12}\intr\frac{y_1y}{(\hat\la+|y|^2)^3} dy
\\ & =8(N+1) \hat\la^{\frac12}\intr\frac{z_1z}{(1+|z|^2)^3} dz=2(N+1)\pi \hat\la^{\frac12}\left(\begin{aligned}&1\\ &0\end{aligned}\right)
 \end{aligned}$$
since $\intr \frac{z_1^2}{(1+|z|^2)^3}dz=\frac\pi4$. 
Moreover,  arguing analogously 
$$
\intr \Im(x^{N+1})|x|^{2N}e^{W_\la} Z_\la dx=2(N+1)\pi \hat\la^{\frac12}\left(\begin{aligned}&0\\ &1\end{aligned}\right) 
  .$$
Similarly
$$\begin{aligned} &\intr \Re\big(x^{2(N+1)}\big)|x|^{2N}e^{W_\la} Z_\la dx=4(N+1)\pi \hat\la^{\frac12}\overline\xi, 
\\ &\intr \Im\big(x^{2(N+1)}\big)|x|^{2N}e^{W_\la} Z_\la dx= 4(N+1)\pi \hat\la^{\frac12}{\rm i} \overline\xi 
.\end{aligned}$$ 
\end{rmk}

\begin{prop}\label{auxauxaux} Let $N\geq 1$. Assume  that, along with \eqref{eqq1},   there exists $i_0\in\N$ with $i_0>  N+1$ such that \eqref{mangia} holds and $V\in C^{2i_0+1}(U)$ in a neighborhood of $0$.
Then,  for any fixed $p\geq 1$ the following holds:
\beq\label{cese111}\|R_\la\|_{p}\leq C\la^{\frac12+\frac{1}{2(N+1)}}
\la^{-\frac{p-1}{(N+1)p}} \eeq 

Moreover: 

\begin{itemize}[leftmargin=*]
\item if $N+1< i_0< 2N+2$ one has 
 \beq\label{dublend60}\into R_\la PZ_\la dx
 + 4\pi(N+1)\hat\la^\frac12 (\beta_{2N+2}\overline\xi-2\xi)
 =O\big(\la^{1+\frac{1}{2(N+1)}}\big)+\sum_{i=i_0+1}^{2N+2}\gamma_{i,N+1} O\big(\la^{\frac{i}{2(N+1)}}\big)
 \eeq
 uniformly for $|\xi|\leq M  \hat\la^{\frac12}$;

\

\item if $i_0\geq 2N+2$ one has 

 \beq\label{dublend6}\into R_\la PZ_\la dx
 - 8(N+1)\pi \hat\la^\frac12\xi=O\big(\la^{1+\frac{1}{2(N+1)}}\big)\eeq
 uniformly for $|\xi|\leq M\hat\la^\frac12$.
 
 \end{itemize}

 \end{prop}

\begin{proof}
\eqref{cese111} follows directly by Lemma \eqref{aux00} and \eqref{litilitio} once we have observed that $V(x)-1=O(|x|^{i_0+1})=O(|x|^{N+2})$. 
Moreover, using 
Lemma \ref{semierror} we deduce 

\beq\label{depipi}\begin{aligned}&R_\la+\big(2\hat\la-4\Re(\overline\xi x^{N+1})\big)|x|^{2N}e^{W_\la}+\big(V(x)-1\big)
|x|^{2N}e^{W_\la}\\ &=\big(|\xi|^2|x|^{2(N+1)}+\la|\xi|+\la^2+
\la|x|^{N+2}+|\xi||x|^{2N+3}\big)O\big(|x|^{2N}e^{W_\la}\big)\end{aligned} \eeq
 uniformly for $|\xi|\leq M\hat\la^\frac12$.
By Lemma \ref{copy0} and Remark \ref{enne=4} we have
$$\intr |x|^{2N}e^{W_\la}Z_\la dx=0,$$ 
$$\intr \Re(\bar\xi x^{N+1}) |x|^{2N}e^{W_\la}Z_\la dx=2\pi(N+1)\hat\la^\frac12\xi.$$

Therefore,  by  testing the above identity \eqref{depipi} against $ Z_\la$ and using  Lemma \ref{santa00},
\beq\label{igor}\begin{aligned} &\into R_\la Z_\la dx-8\pi(N+1)\hat\la^\frac12\xi
+ \into \big(V(x)-1\big)
|x|^{2N}e^{W_\la }  Z_\la dx=O(\la^\frac32)
\end{aligned}\eeq uniformly for $|\xi|\leq M  \hat\la^{\frac12}$. Combining \eqref{pzi} with \eqref{cese111} we obtain
\beq\label{lavi}\into R_\la PZ_\la dx= \into R_\la Z_\la dx+O(\la^\frac12)\into |R_\la| dx= \into R_\la Z_\la dx+O\big(\la^{1+\frac{1}{2(N+1)}}\big)\eeq
uniformly for $|\xi|\leq M  \hat\la^{\frac12}$.

\

If $i_0\geq 2N+2$, then $V(x)-1=O(|x|^{2N+3}) $, and \eqref{dublend6} follows by \eqref{igor} and \eqref{lavi} applying Lemma \ref{santa00}.
\

Next assume that  $N+1< i_0\leq 2N+1$: 
using  again Lemma \ref{santa00} together with 
the expansion $V(x)-1=\sum_{i=i_0+1}^{2N+2}{\rm q}_i(x)+O(|x|^{2N+3})$
\beq\label{imbesi}\begin{aligned}\into|x|^{2N} (V(x)-1)e^{W_\la }  Z_\la dx&= \sum_{i=i_0+1}^{2N+2}\intr|x|^{2N} {\rm q}_i(x)e^{W_\la }  Z_\la dx+O\big(\la^{1+\frac{1}{2(N+1)}}\big).\end{aligned}
\eeq
So, recalling the definition of $  {\rm q}_i$ in \eqref{homo1} and by applying Lemma \ref{copy0} with  $N+1<\ell<2(N+1)$,
 \beq\label{santa3333}\begin{aligned}&\sum_{i=i_0+1}^{2N+2}\intr|x|^{2N} {\rm q}_i(x)e^{W_\la }  Z_\la dx\\& =
 \intr \Re\big(\overline \beta_{2N+2}x^{2N+2}\big)|x|^{2N} e^{W_\la }   Z_\la dx+\sum_{i=i_0+1}^{2N+2}\intr\Re(\overline \gamma_{i,N+1}x^{N+1}) |x|^{N+i-1}e^{W_\la}Z_\la dx\\ &=4\pi (N+1) \hat\la^{\frac12}\beta_{2N+2}\overline\xi+\sum_{i=i_0+1}^{2N+2}\gamma_{i,N+1} O\big(\la^{\frac{i}{2(N+1)}}\big)  \end{aligned}\eeq according to Remark  \ref{enne=4}. 
By inserting \eqref{santa3333} in \eqref{imbesi}, and next combining with  \eqref{igor} and \eqref{lavi}, we obtain \eqref{dublend60}. 
\end{proof} 

{\Large{\section{The nonlinear problem: a contraction argument}}}
In order to find a blow-up sequence for problem \eqref{eq0}-\eqref{eqq0}, let us consider the following intermediate problem:

\beq\label{inter}\left\{\begin{aligned}&-\Delta(PW_\la+\phi)-\la V(x) |x|^{2N}e^{PW_\la+\phi}=\sum_{j=1,2}c_j Z_\la^j|x|^{2N} e^{W_\la},\\ &\phi \in H^1_0(B_1),\;\;\;\; \into \nabla \phi\nabla PZ_\la^jdx=0 \; \; j=1,2.\end{aligned}\right.\eeq
 
 As a first step, it is convenient to solve the problem for $\phi$ as a function of $\xi$. 
 For this purpose, we rewrite  \eqref{inter} in a more convenient form. 
 
 In what follows we denote by $N:K^\perp\to K^\perp$ the nonlinear operator 
$$N(\phi)=\Pi^\perp\({i^*_{p}}\big(       \la V(x)|x|^{2N}e^{PW_\la}(e^{\phi}-1-    \phi) \big)\).$$
 Therefore problem \eqref{inter} turns out to be equivalent to the problem
 
 \beq\label{interop}  L(\phi)+N(\phi)=\tilde R,\quad \phi\in K^\perp\eeq
 where, recalling \eqref{errelambda},  $$\tilde R=\Pi^\perp\({i^*_{p}}\big(R_\la\big)\)=
 \Pi^\perp\(PW_\la -{i^*_{p}}\big(\la V(x)|x|^{2N}e^{P W_\la}\big)\).$$
 
 We need the following auxiliary lemma.
 
 \vspace{2mm}
 \begin{lemma}\label{auxnonl}
For any $p> 1$ there exists $\la_0>0$ such that   for any $\la\in (0,\la_0)$, any $\xi\in \R^2$ with $|\xi|\leq  M\hat\la^\frac12$  and any $\phi_1,\phi_2\in H_0^1(B_1)$ with $\|\phi\|_1,\,\|\phi_2\|<1$ the following holds
\beq\label{skate1}\|e^{\phi_1}-\phi_1-e^{\phi_2}+\phi_2\|_p\leq C\big(\|\phi_1\|+\|\phi_2\|\big)\|\phi_1-\phi_2\|,\eeq
\beq\label{skate11}\Big\|e^{\phi_1}-\phi_1-\frac12\phi_1^2-e^{\phi_2}+\phi_2+\frac12\phi_2^2\Big\|_p\leq C\big(\|\phi_1\|^2+\|\phi_2\|^2\big)\|\phi_1-\phi_2\|,\eeq
 \beq\label{skate2}\|N(\phi_1)-N(\phi_2)\|\leq C\la^{\frac{1-p^2}{ (N+1)p^2}}\big(\|\phi_1\|+\|\phi_2\|\big)\|\phi_1-\phi_2\|.\eeq
 
 \end{lemma}
 \begin{proof}
A straightforward computation gives that  the inequality $|e^a-a-e^b+b|\leq e^{|a|+|b|}(|a|+|b|)|a-b|$ holds  for all $a,b\in \R$. Then, by applying H\"older's inequality with $\frac1q+\frac1r+\frac1t=1$, we derive
$$\|e^{\phi_1}-\phi_1-e^{\phi_2}+\phi_2\|_p\leq C\|e^{|\phi_1|+|\phi_2|}\|_{pq}\big(\|\phi_1\|_{pr}+\|\phi_2\|_{pr}\big)\|\phi_1-\phi_2\|_{pt}$$
 and \eqref{skate1} follows by using Lemma \ref{tmt} and the continuity of the embeddings $H^1_0(B_1)\subset L^{pr}(B_1)$ and $H^1_0(B_1)\subset L^{pt}(B_1)$. 
  Similarly, we now use the inequality $|e^a-a-\frac12a^2-e^b+b+\frac12b^2|\leq e^{|a|+|b|}(|a|^2+|b|^2)|a-b|$ which holds  for all $a,b\in \R$ and we obtain 
 $$\Big\|e^{\phi_1}-\phi_1-\frac12\phi_1^2-e^{\phi_2}+\phi_2+\frac12\phi_2^2\Big\|_p\leq C\|e^{|\phi_1|+|\phi_2|}\|_{pq}\big(\|\phi_1\|_{2pr}^2+\|\phi_2\|_{2pr}^2\big)\|\phi_1-\phi_2\|_{pt}$$
 and we obtain \eqref{skate11}. 
 
 Let us prove \eqref{skate2}. According to \eqref{isp}  we get
 $$\|N(\phi_1)-N(\phi_2)\|\leq C\|\la V(x)|x|^{2N} e^{PW_\la}(e^{\phi_1}-\phi_1-e^{\phi_2}+\phi_2)\|_p,$$
and by H\"older's inequality with $\frac1p+\frac1q=1$ we derive 
$$\begin{aligned}\|N(\phi_1)-N(\phi_2)\|&\leq C\|\la |x|^{2N}e^{PW_\la}\|_{p^2}\|e^{\phi_1}-\phi_1-e^{\phi_2}+\phi_2|\|_{pq}\\ &\leq C\|\la |x|^{2N}e^{PW_\la}\|_{p^2}\big(\|\phi_1\|+\|\phi_2\|\big)\|\phi_1-\phi_2\|\end{aligned}
$$ by \eqref{skate1}, and  the conclusion follows recalling Lemma \ref{aux00}. 
\end{proof}
Problem \eqref{inter} or, equivalently, problem \eqref{interop}, is solvable for any  $\xi$ with $|\xi|\leq M\hat\la^\frac12$, provided
that $\la$ is sufficiently small. Indeed, we state the following  result, which  is divided into two parts corresponding to the two settings described in Propositions  \ref{auxaux} and  \ref{auxauxaux}.

\begin{prop}\label{nonlaux}  Let $N\geq 1$. Assume  that, along with \eqref{eqq1},   there exists $i_0\in\N$ with $ i_0\geq1$ such that $V\in C^{2i_0+1}(U)$ in a neighborhood of $0$  and \eqref{mangia} holds.  Then, for any $\e\in (0,\frac{1}{2(N+1)})$ there exists $\la_0>0$ such that for any $\la\in (0,\la_0)$ and any $\xi\in \R^2$ with $|\xi|<M\hat\la^\frac12$ there is a unique $\phi_\la=\phi_{\la,\xi}\in K^\perp$ satisfying \eqref{inter} for some $c_1,c_2\in \R$ and $$\|\phi_{\la}\|\leq \left\{\begin{aligned}& \la^{\frac{i_0+1}{2(N+1)}}\la^{-\e} &\hbox{ if }&i_0\leq N+1\\ & \la^{\frac12+\frac{1}{2(N+1)}}\la^{-\e}& \hbox{ if }& i_0>N+1\end{aligned} \right. .$$ \end{prop}

\begin{proof} Let us consider, for instance,  the case $i_0\leq N+1$. Since, as we have observed, problem \eqref{interop} is equivalent to problem \eqref{inter}, 
we will show that problem \eqref{interop} can be solved via a contraction mapping argument. Indeed, in virtue of Proposition \ref{ex}, let us introduce the map
$$T:=L^{-1}(\tilde R-N(\phi)),\quad \phi\in K^\perp.$$
 Let us fix $$0<\eta<\e$$ and $p>1$ sufficiently close to 1. According to \eqref{isp} and Proposition \ref{auxaux} we have
\beq\label{non1}\|\tilde R\|\leq C\la^{\frac{i_0+1}{2(N+1)}}\la^{-\eta} .\eeq
Similarly, by \eqref{skate2}, 
\beq\label{non2}\|N(\phi_1)-N(\phi_2)\|\leq C\la^{-\eta}(\|\phi_1\|+\|\phi_2\|)\|\phi_1-\phi_2\|\quad \forall \phi_1,\phi_2\in H_0^1(B_1), \|\phi_1\|,\|\phi_2\|<1.\eeq In particular, by taking $\phi_2=0$, 
\beq\label{non3}\|N(\phi)\|\leq C\la^{-\eta}\|\phi\|^2\quad \forall \phi\in H_0^1(B_1), \|\phi\|<1.\eeq

We claim that $T$ is a contraction map over the ball $$\Big\{\phi\in K^\perp\,\Big|\, \|\phi\|\leq \la^{\frac{i_0+1}{2(N+1)}}\la^{-\e}\Big\}$$ provided that $\la$ is small enough. Indeed, combining  Proposition \ref{ex}, \eqref{non1}, \eqref{non2}, \eqref{non3} with the choice of $\eta$,
we have 
$$\|T(\phi)\|\leq C|\log\la|(\la^{\frac{i_0+1}{2(N+1)}}+\la^{-\eta}\|\phi\|^2)<\la^{\frac{i_0+1}{2(N+1)}}\la^{-\e},$$
$$\begin{aligned}\|T(\phi_1)-T(\phi_2)\|&\leq C|\log\la|\|N(\phi_1)-N(\phi_2)\|\leq C\la^{-\eta}|\log\la| (\|\phi_1\|+\|\phi_2\|)\|\phi_1-\phi_2\|<\frac12\|\phi_1-\phi_2\|.\end{aligned}$$

The case $i_0> N+1$ can be handled  similarly, with Proposition~\ref{auxauxaux} replacing Proposition~\ref{auxaux}.
\end{proof}

{\Large{\section{Proof of Theorems \ref{th4}, \ref{th5}, \ref{mest}.}  }}

After problem \eqref{inter} has been solved according to Propositions \ref{nonlaux}, then we find a solution to the original problem \eqref{eq0}--\eqref{eqq0} if $\xi$   is such that $$c_j=0\hbox{ for }j=1,2.$$ 
 Let us find the condition satisfied by $\xi$ in order to get the $c_j$'s equal to zero. 
We multiply the equation in \eqref{inter} by   $PZ_\la^j$ 
and integrate over $B_1$:
\beq\label{masca}\begin{aligned}&\into \nabla (PW_\la+\phi_{\la}) \cdot\nabla PZ_\la^jdx -\la \into V(x) |x|^{2N}e^{P W_\la+\phi_{\la}}PZ_\la^j dx\\&=\sum_{h=1,2}c_h \into Z_\la^h|x|^{2N} e^{W_\la}PZ_\la^j dx.\end{aligned}
\eeq 
The object  is now to expand each integral of the above identity and analyze the  leading term.
 Let us begin by observing that  the orthogonality in \eqref{inter} gives
\beq\label{masca1} \into \nabla \phi_{\la}\cdot \nabla PZ_\la^j dx =
0\eeq
and, using \eqref{pzi}, 
by a direct computation we obtain \beq\label{insuff}\begin{aligned}&\into Z_\la^h|x|^{2N} e^{W_\la}PZ_\la^j dx
= \intr Z_\la^jZ_\la^h|x|^{2N} e^{W_\la}dx+o(1)=\left\{\begin{aligned}&\frac23 (N+1){\pi}+o(1) &\hbox{ if }h=j\\ &o(1)&\hbox{ if }h\neq j\end{aligned}\right..\end{aligned}\eeq
Using \eqref{errelambda}
we get
\beq\label{coll}\begin{aligned}&\int_{B_1} \nabla PW_\la \cdot\nabla PZ^j_{\la} dx -\la \int_{B_1} V(x) |x|^{2N}e^{P W_\la}PZ^j_{\la} dx=
\int_{B_1} R_\la PZ^j_{\la} dx.\end{aligned}\eeq 
By inserting \eqref{masca1}--\eqref{coll} into \eqref{masca} we obtain
\beq\label{luis}\into R_\la PZ_\la^j=\frac23(N+1)\pi c_j+o(c_1)+o(c_2)+ \la \into |x|^{2N} V(x) e^{P W_\la}(e^{\phi_\la}-1)PZ^j_{\la} dx.\eeq
Let us pass to examine  the term containing the function $\phi_\la$: 
\beq\label{mos}\begin{aligned}\la \into |x|^{2N} V(x) e^{P W_\la}(e^{\phi_\la}-1)PZ^j_{\la} dx&=\la \into |x|^{2N} V(x) e^{P W_\la}(e^{\phi_\la}-1-\phi_\la)PZ^j_{\la} dx \\ &\;\;\;\;-\into R_\la \phi_\la PZ_\la^j dx  +\into |x|^{2N}e^{W_\la}\phi_\la  PZ_\la^j dx.\end{aligned}\eeq
Now, let us fix $\e>0$ sufficiently small and $p>1$ sufficiently close to 1.  
Next let  $1<q<\infty$ be such that  $\frac1p+\frac1q=1$. Then, 
\eqref{skate1} with $\phi_2=0$ and Proposition \ref{nonlaux} give
$$\|e^{\phi_\la}-1-\phi_\la\|_q\leq  C\|\phi_\la\|^2 
$$
and, consequently, 
$$\|e^{\phi_\la}-1\|_q\leq C\|\phi_\la\|
.$$
Therefore,  Lemma \ref{aux00}  implies
 \beq\label{capra1}\begin{aligned} 
 \la\int_{B_1} |x|^{2N} V(x) e^{P W_\la}(e^{\phi_\la}-1-\phi_\la)PZ^j_{\la} dx
 &= O(\| \la |x|^{2N}e^{ P W_\la}(e^{\phi_\la}-1-\phi_\la)\|_1)\\ &=
O(\| \la  |x|^{2N}e^{P W_\la} \|_p\|e^{\phi_\la}-1-\phi_\la\|_q)\\ &=O\big(\la^{-\frac{p-1}{2p}} \|\phi_\la\|^2\big)
\end{aligned}\eeq uniformly for $|\xi|\leq M \hat\la^\frac12$. 
Similarly, \beq\label{olifis2}\begin{aligned}\into R_\la\phi_\la PZ_\la^j dx  &=O\big(\|R_\la\|_p\|\phi_\la\|_q\big)=O\big(\|R_\la\|\|\phi_\la\|\big)\end{aligned}\eeq
uniformly for $|\xi|\leq M \hat\la^\frac12$. 
Finally, recalling the orthogonality \eqref{masca1}, combining Lemma \ref{pziraff} and Lemma \ref{aux00},  
\beq\label{olifis3}\begin{aligned}\into|x|^{2N}e^{W_\la}\phi_\la PZ_\la^j dx&=\into |x|^{2N}e^{W_\la}|\phi_\la|O(\la^\frac12|x|^{N+1}+\la)dx
\\ &=O\big(\|(\la^\frac12|x|^{N+1}+\la)|x|^{2N}e^{W_\la})\|_p\|\phi_\la\|_q\big)=O(\la^{1-\frac{p-1}{p(N+1)}}|\phi_\la\|\big)\end{aligned}
\eeq uniformly for $|\xi|\leq M \hat\la^\frac12$. 
 By inserting \eqref{capra1}--\eqref{olifis3} into \eqref{mos}, and next into \eqref{luis}, we arrive at
 \beq\label{torcer}\begin{aligned}\into R_\la PZ_\la^jdx=&\frac23(N+1)\pi c_j+o(c_1)+o(c_2)\\ &+O\Big(\la^{-\frac{p-1}{2p}}\|\phi_\la\|^2+\|R_\la\|\|\phi_\la\|+\la^{1-\frac{p-1}{p(N+1)}}\|\phi_\la\| \Big)\end{aligned}\eeq
 for $j=1,2$, uniformly for $|\xi|\leq M \hat\la^\frac12$. 
 We now conclude separately the proofs of Theorems \ref{th4}, \ref{th5}, \ref{mest}.
 
 \vspace{5mm}
 
\subsection{Proof of Theorem \ref{th4}.} 
Assume that $N\geq 1$ and that, along with \eqref{eqq1}, there exists $i_0\geq 1$  
such that $V$ is of class $C^{2i_0+1}(U)$ in a neighborhood $U$ of $0$ and assumption \eqref{mangia} holds. Equivalently, by \eqref{homo1} and Lemma \ref{finallemma}, $i_0$ corresponds to the first index such that $\alpha_{2i_0}\neq 0$ and 
$${\rm q}_i=0\quad \forall 1\leq i\leq i_0.$$
 First consider the case $i_0\leq N+1$. Then, according to Proposition \ref{auxaux} and 
 Proposition \ref{nonlaux}
\beq\label{valc}
\|\phi_{\la}\|\leq \la^{\frac{i_0+1}{2(N+1)}}\la^{-\e} ,\qquad \|R_\la\|\leq C  \la^{\frac{i_0+1}{2(N+1)}}\la^{-\frac{p-1}{p(N+1)}}\eeq 
uniformly for $|\xi|\leq M  \hat\la^{\frac12}$, 
so that 
 \eqref{torcer} takes the vector form\beq\label{ario}\into R_\la PZ_\la dx=\frac23(N+1) \pi c+o(c)+O
\big(\la^{\frac{2i_0+1}{2(N+1)}}
\big)\eeq uniformly for $|\xi|\leq M  \hat\la^{\frac12}$, 
 provided that $\e>0$ is sufficiently small and $p>1$ is sufficiently close to $1$.
 \

First consider the  two cases:  

\begin{itemize}
\item  $N\geq 2i_0$, 

\item  $N+1<2i_0<2(N+1)$ and the Taylor expansion of $V$ at $0$ between degrees $N+1$ and $2i_0$ has vanishing angular Fourier mode of order $N+1$.

\end{itemize}
Equivalently, using \eqref{homo1}, suppose that, either

$$N\geq  2 i_0$$ 
or  $$N+1<2i_0<2(N+1)\hbox{ and }\beta_{N+1}=\gamma_{i,N+1}=0\hbox{ for }i=N+2,\ldots, 2i_0$$
so that combining \eqref{ario} with  \eqref{dublend1}, \eqref{dublend2} we obtain 
\beq\label{ttcc} \alpha_{2i_0} \Gamma_{2i_0}\hat\la^{\frac{i_0}{N+1}-\frac12}\xi+O\big(\la^{\frac{2i_0+1}{2(N+1)}}+
|\xi|^2\la^{\frac{i_0}{N+1}-1}\big)
=-\frac23(N+1)\pi c+o(c),\eeq uniformly for $|\xi|\leq M  \hat\la^{\frac12}$.  So, setting $$ \xi=\hat\la^{\frac12+\frac{1}{3(N+1)}} \tilde \xi$$ we rewrite \eqref{ttcc} as \beq\label{tccone1x}\alpha_{2i_0}\Gamma_{2i_0}\la^{\frac{3i_0+1}{3(N+1)}}\tilde\xi+o\big(\la^{\frac{3i_0+1}{3(N+1)}} \big)= -\frac23(N+1)\pi c+ o(c)\hbox{ unif. for }|\tilde\xi|\leq M.\eeq

The continuity of the map $\xi\mapsto \phi_\la=\phi_{\la,\xi}$, guaranteed by Proposition \ref{nonlaux},  implies that the left hand
side of \eqref{tccone1x} is continuous too.  
Since $0$ is a zero for the identity map which is stable under uniform perturbations, and  $\alpha_{2i_0}\neq 0$, then the uniform stability gives that, fixed $0<\eta<1$,
 for $\la$ small enough the left hand side of \eqref{tccone1x} has a zero $\tilde \xi_\la$ with $|\tilde \xi_\la|\leq \eta$ or, equivalently,
the left hand side of \eqref{ttcc} has a zero $\xi_\la$ with $|\xi_\la|\leq \hat\la^{\frac12+\frac{1}{3(N+1)}}\eta $. So for such $\xi_\la$  the linear system \eqref{ttcc} has only the trivial solution $c_1=c_2=0$. The arbitrariness of $\eta$ implies $$\xi_\la =o(\la^{\frac12+\frac{1}{3(N+1)}
}).$$ 
 Then, once a solution of problem \eqref{inter} with 
$c_j=0$ has been found, we have also obtained a solution to the original problem \eqref{eq0}--\eqref{eqq0}
of the form 
 $$v_\la=PW_\la+\phi_\la$$   for some $\xi=\xi_\la=o(\la^\frac12
 )$ and $\phi_\la$ satisfying \eqref{valc}.
 We claim that $\phi_\lambda$ is uniformly small in $B_1$. Indeed, $\phi_\lambda$ satisfies
$$
-\Delta \phi_\lambda
=
\lambda V(x)|x|^{2N}e^{PW_\lambda}
\bigl(e^{\phi_\lambda}-1\bigr)
-R_\lambda.
$$
Fix $p>1$ sufficiently close to $1$ and let $1<q<\infty$ such that $\frac1p+\frac1q=1$. By H\"older's inequality, Lemma~\ref{aux00},
Lemma~\ref{auxnonl} (with $\phi_1=\phi_\lambda$ and $\phi_2=0$), and \eqref{valc}, we obtain
$$
\begin{aligned}
\|\Delta \phi_\lambda\|_p
&\leq C\|\lambda |x|^{2N}e^{PW_\lambda}\bigl(e^{\phi_\lambda}-1\bigr)+R_\la\|_p\\ &
C\|\lambda |x|^{2N}e^{PW_\lambda}\|_{p^2}
\left(
\|e^{\phi_\lambda}-1-\phi_\lambda\|_{pq}
+\|\phi_\lambda\|_{pq}
\right)
+\|R_\lambda\|_p\\
&\leq
C\lambda^{-\frac{p^2-1}{p^2(N+1)}}
\left(
\|\phi_\lambda\|^2+\|\phi_\lambda\|
\right)
+\|R_\lambda\|_p\\
&\leq C\lambda^\sigma
\end{aligned}
$$
for some $\sigma>0$, provided that $p>1$ is sufficiently close to $1$ and
$\varepsilon>0$ is sufficiently small. Therefore, by Dirichlet elliptic regularity and the embedding
$
W^{2,p}(B_1)\hookrightarrow L^\infty(B_1)$,
 $p>1,
$
we conclude that
$$
\|\phi_\lambda\|_\infty=o(1).
$$
  We also have ${\cal H}(x^{N+1},\xi_\la)=o(1)$ and $|x^{N+1}-\xi_\la|^2=|x|^{2(N+1)}(1+o(1))+o(\la)$ uniformly in $B_1$, by which, using 
  \eqref{pro-exp1}, $$v_\la= -2\log\big(\hat\la+|x|^{2(N+1)}\big)
+o(1).$$ 
Moreover, by concavity we get  $\log(\hat\la+|x|^{2(N+1)})\geq \frac12\log(2\hat\la)+\frac12\log (2|x|^{2(N+1)})$ and, consequently, 

$$ \begin{aligned}v_\la(x) +2(N+1)\log|x|+\log\la&\leq -\log(2\hat\la)-\log (2|x|^{2(N+1)}) +2(N+1)\log|x|+\log\la+o(1)\\ &=\log(2(N+1)^2)+o(1)\end{aligned} $$which implies that $v_\la$ exhibits simple blow-up in the sense of \eqref{simple}. This concludes the proof of Theorem \ref{th4} when $N\geq 2i_0$ or $N+1<2i_0<2(N+1)$. 

\

We pass to consider the case when $i_0=N+1$, the Taylor polynomial of $V$ at $0$ formed by the homogeneous terms of degrees $N+1,\dots,2i_0$ has no angular Fourier component of order $N+1$, and  $ \Delta^{i_0} V(0) \neq4^{i_0}(i_0!)^2!(2\pm|\beta_{2i_0}|)$; in particular, using \eqref{homo1} and  Lemma       \ref{finallemma}, this implies $$i_0=N+1, \;\;\gamma_{i,N+1}=0\;\;\forall i=N+2,\ldots, 2N+2,\hbox{ and }\alpha_{2N+2}-2\neq \pm|{\beta}_{2N+2}|.$$ 
Under the assumptions of Theorem \ref{th4}, $\beta_{2N+2}=0$, so the above nondegeneracy condition reduces to $\alpha_{2N+2}\neq2$.
In view of  \eqref{dublend4}, \eqref{ario} can be rewritten as 

\beq\label{titti}2(\alpha_{2N+2}\xi+{\beta}_{2N+2}\overline\xi-2\xi)\hat\la^\frac12 +O(\la^{1+\frac{1}{2(N+1)}})=-\frac{c}{3} + o(c)\eeq uniformly for $|\xi|\leq M\hat\la^\frac12$. 
Setting $$ \xi=\hat\la^{\frac12+\frac{1}{3(N+1)}} \tilde \xi$$ we rewrite \eqref{titti} as
\beq\label{tatta}2(\alpha_{2N+2}\tilde\xi+{\beta}_{2N+2}\overline{\tilde \xi}-2\tilde\xi)\hat\la^{1+\frac{1}{3(N+1)}} +o\big(\la^{1+\frac{1}{3(N+1)}} \big)= -\frac{c}{3} + o(c)\hbox{ unif. for }|\tilde\xi|\leq M.\eeq
 Since the map $$\xi\mapsto\alpha_{2N+2}\xi+{\beta}_{2N+2}\overline\xi -2\xi$$ has determinant $$(\alpha_{2N+2}-2)^2-|{\beta}_{2N+2}|^2,$$ it follows that, if $\alpha_{2N+2}-2\neq\pm|{\beta}_{2N+2}|$ then 0 is a nondegenerate (and hence stable) zero; so proceeding in the same way as for  \eqref{tccone1x}, and fixing an arbitrary $0<\eta<1$, we find that the left-hand side of \eqref{tatta} admits a zero $\tilde \xi_\la$ such that $|\tilde\xi_\la|<\eta$. Equivalently, the left-hand side of \eqref{titti} has a zero 
$\xi_\la$ satisfying $|\xi_\la|\leq \eta \hat\la^{\frac12+\frac{1}{3(N+1)}} $. For this value of $\xi_\la$, the corresponding linear system \eqref{titti} has only the trivial solution 
$c_1=c_2=0$. Since $\eta$ was arbitrary, we conclude that $$\xi_\la=o\big( \la^{\frac12+\frac{1}{3(N+1)}}\big) .$$
In particular, $\xi_\la=o(\la^\frac12)$, and we can therefore proceed exactly as for the previous case.

 \
 
 Next we turn to the proof of Theorem \ref{th4} in the complementary case $i_0> N+1.$ We now apply Proposition \ref{auxauxaux} in combination with Proposition \ref{nonlaux} and obtain \beq\label{valch}\|\phi_{\la}\|\leq \la^{\frac12+\frac{1}{2(N+1)}}\la^{-\e} ,\qquad \|R_\la\|\leq C  \la^{\frac12+\frac{1}{2(N+1)}}\la^{-\frac{p-1}{p(N+1)}}\eeq 
uniformly for $|\xi|\leq M  \hat\la^{\frac12}$. Consequently, \eqref{torcer} reads:
\beq\label{tisi}\into R_\la PZ_\la dx =\frac23 \pi(N+1)c+o(c)+
O
\big(\la^{1+\frac{1}{2(N+1)}}\big)\eeq uniformly for $|\xi|\leq M  \hat\la^{\frac12}$, 
 provided that $\e>0$ is sufficiently small and $p>1$ is sufficiently close to $1$.
 
Let us distinguish the cases $N+1< i_0<2N+2$ and $i_0\geq 2N+2$. 
 
 First, let us focus on the case when $N+1<i_0<2(N+1)$, the homogeneous terms in  the Taylor polynomial of $V$ at $0$ of degrees between $N+1$ and $2i_0$
have zero angular Fourier component of order $N+1$, and $|\beta_{2N+2}|\neq2$; in other words, 

 $$N+1<i_0<2(N+1)\;\;\;\;\gamma_{i,N+1}=0\;\;\forall i=i_0+1,\ldots, 2N+2,\hbox{ and }\;\;|\beta_{2N+2}|\neq 2$$
 so that, according to \eqref{dublend60}, \eqref{tisi} becomes
 \beq\label{tuttu}2({\beta}_{2N+2}\overline\xi-2\xi)\hat\la^\frac12 +O(\la^{1+\frac{1}{2(N+1)}})=-\frac{c}{3}+ o(c)\eeq
 uniformly for $|\xi|\leq M\hat\la^\frac12$. 
This equation  can then be treated exactly as \eqref{titti}, leading to the existence of $\xi_\la=o(\la^{\frac12+\frac{1}{3(N+1)}})$, for which the corresponding linear system \eqref{tuttu} admits only the trivial solution
$c_1=c_2=0$. The conclusion therefore follows exactly as in the previous cases.

\

Finally, assume $i_0\geq 2N+2$: then \eqref{dublend6} and \eqref{tisi} yield $$4\hat\la^\frac12\xi+O\big(\la^{1+\frac{1}{2(N+1)}}\big)= \frac{c}{3}+o(c)$$ uniformly for $|\xi|\leq M  \hat\la^{\frac12}$. 
This equation can then be treated exactly as \eqref{titti}.

\

\begin{rmk}\label{caseNbet}
Let us observe that if in Theorem \ref{th4} we consider the regime $N+1<2i_0<2(N+1)$ and $\beta_{N+1}\neq 0$, then, in view of \eqref{dublend2},  \eqref{ttcc} takes the form
\beq\label{art}\alpha_{2i_0}\Gamma_{2i_0}\hat\la^{\frac{i_0}{N+1}-\frac12}\xi+2\pi(N+1)\beta_{N+1}\hat\la^\frac12+\hbox{h.o.t}=- \frac23(N+1)\pi c+o(c) 
\eeq uniformly for $|\xi|\leq M\hat\la^\frac12$. 
Since the vector field $$\xi\mapsto \alpha_{2i_0}\Gamma_{2i_0}\hat\la^{\frac{i_0}{N+1}-\frac12}\xi+2\pi(N+1)\beta_{N+1}\hat\la^\frac12$$
has no zeros for $|\xi|\le M \hat\lambda^\frac12$, it follows that the left-hand side of \eqref{art} does not vanish  in this region. Hence the method used in the proof of Theorem \ref{th4} breaks down. 

An analogous argument holds for $\gamma_{i,N+1}\neq 0$ when $i=N+2,\ldots, 2i_0$. 
 
 \end{rmk}

\

\subsection{Proof of Theorem \ref{th5}.} We next prove  Theorem \ref{th5} and  assume that $N\geq 1$ is odd and 
 $$ 2i_0=N+1\;\;\hbox{ and }\;\; {\rm q}_i=0\;\;\forall 1\leq i\leq i_0
 .$$
 
 Then, according to Proposition \ref{auxaux} and 
 Proposition \ref{nonlaux}
\beq\label{valco}
\|\phi_{\la}\|\leq \la^{\frac14+\frac{1}{4i_0}}\la^{-\e} ,\qquad \|R_\la\|\leq C  \la^{\frac14+\frac{1}{4i_0}}\la^{-\frac{p-1}{p(N+1)}}\eeq 
uniformly for $|\xi|\leq M  \hat\la^{\frac12}$, 
so that 
 \eqref{torcer} takes the vector form\beq\label{ariooo}\into R_\la PZ_\la dx=\frac23(N+1) \pi c+o(c)+O
\big(\la^{\frac12+\frac{1}{4i_0}}
\big)\eeq uniformly for $|\xi|\leq M  \hat\la^{\frac12}$, 
 provided that $\e>0$ is sufficiently small and $p>1$ is sufficiently close to $1$.
  So   \eqref{ariooo} and  \eqref{dublend3} imply 
 \beq\label{ttcc1} \begin{aligned}&4  \alpha_{2i_0}\hat\la^\frac12{\cal F}(\hat\la^{-\frac12}\xi)
+\pi  \hat\la^\frac12\beta_{2i_0}
+o(\la^\frac12)=-\frac{c}{3}\pi +o(c)\end{aligned}
\eeq uniformly for $|\xi|\leq M  \hat\la^{\frac12}$. 
Now, setting $$\tilde \xi= \hat\la^{-\frac12} \xi$$
 we rewrite \eqref{ttcc1} as 
 \beq\label{ttcc2} \begin{aligned}&4   \alpha_{2i_0}\hat\la^\frac12{\cal F}(\tilde \xi)
+\pi  \hat\la^\frac12 \beta_{2i_0} +o(\la^\frac12)
=-\frac{c}{3}\pi +o(c)\hbox{ unif. for }|\tilde \xi|\leq M.
\end{aligned}
\eeq 
 According to Theorem \ref{th5} suppose that $|\Delta^{i_0}V(0)|>\kappa_N|\beta_{2i_0}|$, or, equivalently, in view of  Lemma \ref{finallemma}, $|\beta_{2i_0}|<|\alpha_{2i_0}|$. By Lemma \ref{fibrafibrafibra}, let $\xi^*\in\R^2$ be such that
 $${\cal F}(\xi^*)= -\frac{\pi}{4} \frac{\beta_{2i_0}}{\alpha_{2i_0}}.$$ According to Lemma \ref{fibrafibrafibra} we have ${\rm det} \,D{\cal F}(\xi^*)\neq 0$, so $\xi^*$ is a zero for the map $\tilde \xi\mapsto 4   \alpha_{2i_0}{\cal F}(\tilde \xi)+\pi   \beta_{2i_0} $ which is stable under uniform perturbations. Assuming that $M>0$ is sufficiently large such that $|\xi^*|< M $, then the uniform stability gives that, fixed $0<\eta<1$,  for $\la$ small enough the left hand side of \eqref{ttcc2} has a zero $\xi^*_\la$ with $|\xi_\la^*-\xi^*|\leq \eta$ or, equivalently, the left hand side of \eqref{ttcc1} has a zero $\xi_\la$ with $ |\xi_\la- \hat\la^{\frac12}\xi^*|\leq  \hat\la^{\frac12} \eta. $ So, for such $\xi_\la$ the linear system \eqref{ttcc1} has only the trivial solution $c_1=c_2=0$. The arbitrariness of $\eta$ implies
\beq\label{ariaaaa}\xi_\la- \hat\la^{\frac12}\,\xi^*=o(\la^\frac12).\eeq
We may therefore proceed in the same manner as in Theorem \ref{th4}. Then, once a solution of problem \eqref{inter} with 
$c_j=0$ has been found, we have also obtained a solution to the original problem \eqref{eq0}--\eqref{eqq0}  
of the form 
 $$v_\la=PW_\la+\phi_\la$$   for some $\xi=\xi_\la$
  and $\phi_\la$ satisfying \eqref{valco} and \eqref{ariaaaa}. Proceeding as in the proof of Theorem \ref{th4}, we also have $\phi_\la=o(1) $ uniformly in $B_1$. Moreover,  ${\cal H}(x^{N+1},\xi_\la)=o(1)$ and $|x^{N+1}-\xi_\la|^2=|x^{N+1}- \hat\la^{\frac12}\xi^*|^{2}(1+o(1))+o(\la)$. Hence,   using 
  \eqref{pro-exp1}, $$v_\la= -2\log\big(\hat\la+\big|x^{N+1}-\hat\la^\frac12\xi^*\big|^{2}\big)
+o(1).$$ 
Moreover, we have  $\hat\la+|x^{N+1}-\hat\la^{\frac12}\, \xi^*|^2\geq c(\hat\la+ |x|^{2(N+1)})$ with $c=\inf_{\R^2}\frac{1+|y^{N+1}- \xi^*|^2}{1+|y|^{2(N+1)}}$, whereas  by concavity we obtain  $\log(\hat\la+|x|^{2(N+1)})\geq \frac12\log(2\hat\la)+\frac12\log (2|x|^{2(N+1)})$, consequently, 

$$ \begin{aligned}&v_\la(x) +2(N+1)\log|x|+\log\la\\&\leq -2\log c-2\log\big(\hat\la+|x|^{2(N+1)}\big)+2(N+1)\log|x|+\log\la+o(1)\\ &\leq -2\log c-\log\big(2\hat\la\big)-\log (2|x|^{2(N+1)}) +2(N+1)\log|x|+\log\la+o(1)=\log\frac{2(N+1)^2}{c^2}+o(1)\end{aligned}$$ which implies that $v_\la$ exhibits simple blow-up in the sense of \eqref{simple}. 

\

\subsection{Proof of Theorem \ref{mest}.} 
Let us fix an odd integer $N\geq 1$  and construct  a family of  coefficient functions $V_{\delta}$, parametrized with $\delta\in(0,1)$,  satisfying 
$$\max_{|\mu|=\frac{N+3}{2}}|D^\mu V^\delta (0)|\geq \frac1C ,
$$
$$V_\delta(0)=1,\quad \|V_\delta\|_{C^{N+2}(\overline B_1)}\leq C ,\quad \frac1C\leq V_\delta(x)\leq C\quad \forall x\in \overline B_1,$$
$$\frac{|{\beta}_{N+1}^\delta|}{|\alpha_{N+1}^\de|} \to 1^-  \hbox{ as } \delta\to 0^+,$$
for some positive constant $C$ independent of $\de$. Here $\Re(\overline{\beta}_{N+1}^\delta x^{N+1})$ denotes the harmonic component of order $N+1$ in the expansion of $V_\delta$ at the origin, and $\alpha_{N+1}^\de=\frac{1}{\kappa_N}\Delta^{\frac{N+1}{2}}V_\delta(0)$, according to Lemma \ref{finallemma}. An example of such a family is given by 

$$V_\delta(x)=1+\frac14\Re(x^{\frac{N+3}{2}})+\frac{1}{4}|x|^{N+1}+\frac{1-\de}{4}\Re(x^{N+1})\hbox{ if } N\geq 3$$
and $$V_\delta (x)=1+\frac14|x|^2+\frac{1-\de}{4}\Re(x^2)\hbox{ if }N=1.$$
Indeed, for such $V_\de$ we have 

\begin{itemize}

\vspace{2mm }

\item $V_\de(0)=1$, $\nabla V(0)=0$, $\frac12\leq V_\delta (x)\leq \frac74\quad \forall x\in \overline{B}_1$;

\vspace{2mm}

\item $\Delta^{\frac{N+1}{2}} V_\de(0)= \frac{\kappa_N}{4}$ and, if $N\geq 3$,   $  \Delta V_\de(0)=\Delta^2 V_\de(0)=\ldots =\Delta^{\frac{N-1}{2}}V_\de (0)=0$;

\vspace{2mm}

\item $D^\mu V_\de(0)=0$ if $1\leq |\mu|\leq \frac{N+1}{2};$

\vspace{2mm}

\item $\partial _{x_1}^{\frac{N+3}{2}}V_\de(0)=(\frac{N+3}{2})!\cdot \frac14$ for $N\geq 3$ and $\partial _{x_1}^2V_\delta(0)=1-\frac\de2$ if $N=1$;

\vspace{2mm}

\item $\|V_\delta\|_{C^m(\overline B_1)}\leq C \quad \forall m\in\N$;

\vspace{2mm}

\item${\beta}_{N+1}^\delta=\frac{1-\de}{4}$ and $\alpha_{N+1}^\de=\frac14$;  consequently
$\frac{|{\beta}_{N+1}^\delta|}{|\alpha_{N+1}^\de|}\to 1^- \hbox{ as }\de\to 0^+$.
\end{itemize}

\vspace{2mm}

Therefore, for each fixed $\delta>0$, we may apply Theorem \ref{th5} with $V=V_\delta$, obtaining   a family of simple blow-up solutions to problem \eqref{eq0}--\eqref{eqq0} of the form 
$$v_\la=-2\log\big(\hat\la+\big|x^{N+1}- \hat\la^{\frac12}\xi^*\big|^2\big)+o(1)$$   with 
$${\cal F}(\xi^*)=-\frac\pi4\frac{\beta_{N+1}^\de}{\alpha_{N+1}^\de}.$$
Notice that both the family $v_\la $ and the vector $\xi^*$ depend on the chosen parameter $\de>0$. To emphasize this dependence, we denote them by $v_\la^\de$ and  $\xi_\de^*$, respectively.

Let $\de_k\to 0^+$ be any sequence of positive numbers. Then there exists  a decreasing sequence  $\la_k\to 0^+$ such that  $v_{\la_k}^{\de_k}$ solves \eqref{eq0}--\eqref{eqq0} and $$\big|v_{\la_k}^{\de_k}(x)+2\log\big(\hat\la_k+\big|x^{N+1}-\hat\la_k^\frac12\xi_{\de_k}^*\big|^2\big)\big|\leq \frac{1}{k}.$$
and, in addition,  $$\hat\la_k^\frac12|\xi_{\de_k}^*|<\frac1k.$$
We define $V_k=V_{\de_k}$, $v_k=v_{\la_k}^{\de_k}$ and $\xi_k^*=\xi_{\de_k}^*$. Then we clearly have 
$$|{\cal F}(\xi^*_k)|=\frac\pi4\frac{| \beta_{N+1}^{\de_k}|}{|\alpha_{N+1}^{\de_k}|}\to {\frac\pi4}^{-}$$
which implies, according to Lemma \ref{fibrafibrafibra}, $$|\xi_k^*|\to +\infty.$$
Moreover, choosing 
 $x_k$ to be an $(N+1)$-th root of $\hat\la_k^\frac12\xi_{k}^*$, so that $x_k^{N+1}=\hat\la_k^\frac12\xi^*_k$, we obtain  $x_k\to 0$ and $$\begin{aligned}v_k(x_k) +2(N+1)\log|x_k|+\log\la_k&=2\log |\xi_k^*|+\log(8(N+1)^2)+O\Big(\frac1k\Big)\to +\infty.
\end{aligned} $$Therefore, $v_k$ reveals a non-simple blow-up profile according to \eqref{nons}. This concludes the proof of Theorem \ref{mest}. 

\begin{lemma}\label{fibra} Let $N\geq  1$ and $0\leq i<\frac32(N+1)$. Let $f:\R^2\to\R^2$ be the vector field defined as $$f(\xi)=\intr|z+\xi|^{\frac{2i}{N+1}}\frac{z}{(1+|z|^2)^3}dz.$$
Then $f\in C^1(\R^2)$, $f(0)=0$  and $$D f(0)=\frac{\Gamma_{2i} }{8(N+1)}\left(\begin{aligned}&1&0\\&0&1\end{aligned}\right) $$ where the numbers $\Gamma_{2i}$ have been defined in \eqref{gammai}.  
\end{lemma}

\begin{proof}

Since $f$ is odd, we have $f(0)=0$. We first compute the Jacobian matrix at $\xi=0$. 
The mixed terms vanish:
$$
\partial_{\xi_1}f_2(0)=\partial_{\xi_2}f_1(0)=\frac{2i}{N+1}\int_{\mathbb{R}^2} z_1 z_2 |z|^{2\frac{i-N-1}{N+1}} (1+|z|^2)^{-3} dz = 0,
$$
since the integrand is odd with respect to $z_1$ or $z_2$.
Moreover, by rotational symmetry,
$$
\partial_{\xi_1}f_1(0)=\partial_{\xi_2}f_2(0)=\frac{i}{N+1} \intr |z|^{\frac{2i}{N+1}} \frac{1}{(1+|z|^2)^3}\,dz.$$

\end{proof}

\begin{lemma}\label{fibrafibrafibra} Let  ${\cal F}:\R^2\to\R^2$ be the vector field defined as $${\cal F}(\xi)=\intr|z+\xi|
\frac{z}{(1+|z|^2)^3}dz.$$
Then:

\begin{itemize}
\item ${\cal F}\in C^1(\R^2)$ and ${\rm det}D{\cal F}(\xi)\neq	 0$ for all $\xi\in\R^2$; 
\item ${\cal F}$ is rotation-equivariant and  $|{\cal F}(\xi)|$ is radially symmetric and strictly increasing in $|\xi|$;
\item $\lim_{|\xi|\to +\infty}|{\cal F}(\xi)|=\frac\pi4$;
\item ${\cal F}$ is injective and its image is the open disk of radius  $\frac\pi4$: ${\cal F}(\R^2)=\big\{|\xi|<\frac\pi4\big\}.$

\end{itemize}
\end{lemma}

\begin{proof}
We first compute the Jacobian matrix at $\xi=0$. 
The mixed terms vanish:
$$
\partial_{\xi_1}{\cal F}_2(0)=\partial_{\xi_2}{\cal F}_1(0)=\int_{\mathbb{R}^2} z_1 z_2 |z|^{-1} (1+|z|^2)^{-3} dz = 0,
$$
since the integrand is odd with respect to $z_1$ or $z_2$.
Moreover, by rotational symmetry,
$$
\partial_{\xi_1}{\cal F}_1(0)=\partial_{\xi_2}{\cal F}_2(0)=\frac{1}{2} \intr |z| \frac{1}{(1+|z|^2)^3}\,dz$$
and we immediately get $${\rm det}\, D{\cal F}(0)>0.$$
For every rotation $\theta\in\R$, performing the change of variables $z \mapsto e^{{\rm i}\theta}z$ gives
$$
\mathcal{F}(e^{{\rm i}\theta}\xi) = e^{{\rm i}\theta} \mathcal{F}(\xi).
$$
Hence $\mathcal{F}$ is rotation-equivariant, and therefore there exists a scalar function $\phi:(0,\infty)\to\mathbb{R}$ such that
\beq\label{fib}
\mathcal{F}(\xi) = \phi(|\xi|)\,\xi.
\eeq
By choosing $\xi=(r,0)$ in \eqref{fib}, we write
$$
\phi(r)
=
\frac{1}{r}
\int_{\mathbb{R}^2}
|z+(r,0)|\,\frac{z_1}{(1+|z|^2)^3}\,dz.
$$
Split $\mathbb{R}^2$ into the regions $\{z_1>0\}$ and $\{z_1<0\}$, and perform the change of variables $z\mapsto -z$ in the second  integral. This yields
$$
\phi(r)
=
\frac{1}{r}
\int_{z_1>0}
\frac{z_1}{(1+|z|^2)^3}
\left(
|z+(r,0)| - |z-(r,0)|
\right) dz\quad \forall r>0.
$$
If $z_1>0$, then $|z+(r,0)| > |z-(r,0)|$. 
Hence
$$
\phi(r) > 0 \quad \text{for all } r>0.
$$
Differentiating under the integral sign (justified by dominated convergence), we obtain
$$
\phi(r)+r\phi'(r)=\frac{d}{dr}\big(r\phi(r)\big)
=
\int_{\mathbb{R}^2}
|z+(r,0)|^{-1}
\frac{(z_1+r)z_1}{(1+|z|^2)^3}
\,dz.
$$
Repeating the same symmetry argument as before, this can be rewritten as
$$\phi(r)+r\phi'(r)=
\int_{z_1>0}
\frac{z_1}{(1+|z|^2)^3}
\Big[
|z+(r,0)|^{-1}(z_1+r)
+
|z-(r,0)|^{-1}(z_1-r)
\Big] dz.
$$
Since   $|z+(r,0)|^{-1}(z_1+r)
>
|z-(r,0)|^{-1}(r-z_1)$ if $z_1>0$,  the term in brackets is  positive. Hence  \beq\label{moment}(r\phi(r))'=\phi(r)+r\phi'(r) > 0.\eeq Then by \eqref{fib}  $$
\det D\mathcal{F}(\xi)
=
\phi(|\xi|)\,(\phi(|\xi|)+|\xi|\phi'(|\xi|))
>0
\quad \text{for all } \xi \neq 0.
$$
 
Consequently $|{\cal F}(\xi)|=\phi(|\xi|)|\xi|$ is  strictly increasing in $|\xi|$ by \eqref{moment}. Hence, the rotational equivariance of ${\cal F}$ implies that ${\cal F}$ is injective. 
Moreover, using integration by parts,  $$\begin{aligned}
r\phi(r)&=\intr |z+(r,0)|\frac{z_1}{(1+|z|^2)^3} dz=-\frac14\intr |z+(r,0)|\partial_{z_1}\bigg(\frac{1}{(1+|z|^2)^2}\bigg) dz\\ &=\frac14\intr \frac{z_1+r}{|z+(r,0)|}\frac{1}{(1+|z|^2)^2} dz\longrightarrow  \frac14\intr\frac{1}{(1+|z|^2)^2}dz=\frac\pi4\end{aligned}$$ as $r\to +\infty$, and  therefore
$$\{r\phi(r)\,|\, r\geq 0\}=\Big[0,\frac\pi4\Big).$$  
Finally we  derive $$|F(\xi)|=|\xi|\phi(|\xi|)\longrightarrow  \frac\pi4\hbox{ as }|\xi|\to +\infty$$
and 
the characterization of the image of ${\cal F}$ follows immediately   $${\cal F}(\R^2)=\{ r\phi(r) e^{{\rm i}\theta}\,|\,r\geq 0 , \,\theta\in \R\}=\Big\{\xi\,\Big|\,|\xi|< \frac\pi4\Big\}.$$ 

\end{proof}

\appendix
\renewcommand{\theequation}{\Alph{section}.\arabic{equation}}

\section{}

We compute explicit formulas for the quantities $\Gamma_{2m}$ and collect some integral identities used throughout the paper.  

\begin{lemma} 
\label{finalaux} Let $N,m\in\N$ with $N\geq 1$ and $m\leq N+1$. Then the following holds: $$\Gamma_{2m}=\left\{\begin{aligned}&4\pi^2\frac{m^2(N+1-m)}{(N+1)^2}\frac{1}{\sin(\frac{\pi m}{N+1})}&\hbox{ if }&m\leq N\\ &4\pi m& \hbox{ if }&m=N+1\end{aligned}\right. $$ 
where the numbers $\Gamma_{2m}$ have been defined in \eqref{gammai}.
\end{lemma}

\begin{proof} 
The integrals $\Gamma_{2m}$ belong to the general family $$I_{2m,k,N}=\intr |y|^{\frac{2m}{N+1}}\frac{1}{(1+|y|^2)^k} dy,\quad 0\leq \frac{2m}{N+1}<2(k-1).$$  The general formula for this integral is expressed in terms of the Beta function. 
In particular, for $k=3$ and $m\leq N$
$$I_{2m,3,N}=\intr |y|^{\frac{2m}{N+1}}\frac{1}{(1+|y|^2)^3} dy=\frac{\pi^2}{2}\frac{m(N+1-m)}{(N+1)^2}\frac{1}{\sin(\frac{\pi m}{N+1})}.$$

Taking into account that 
$$\Gamma_{2m}=8m I_{2m,3,N}, 
$$ we obtain the claimed formula for $m\leq N$. For $m=N+1$ the value of $\Gamma_{2m}=\Gamma_{2N+2}$ can be computed by direct integration. 
\end{proof}

\

Next, we derive some integral identities associated with the change of variable $x\mapsto x^{N+1}$ which appears frequently when dealing with $(N+1)$-symmetric functions. \begin{lemma}\label{copy} Let $N\in\N$ and let $\zeta:\R^2\to \R$ be  a continuous function such that   $|x|^{2N} \zeta(x^{N+1})\in L^1(\R^2)$.
Then 
$$\intr |x|^{2N} \zeta(x^{N+1}) dx=\frac{1}{N+1}\intr \zeta(y)dy.$$
\end{lemma}
\begin{proof}We use polar coordinates $x = (\rho,\theta)$ and perform the change of variables $(\rho',\theta')=(\rho^{N+1},(N+1)\theta)$. Then 
$$\begin{aligned} \intr |x|^{2N} \zeta(x^{N+1}) dx&=\int_0^{+\infty} d\rho\int_0^{2\pi} \rho^{2N+1}\zeta(\rho^{N+1} e^{{\rm i}(N+1) \theta}) d\theta
\\ &= \frac{1}{(N+1)^2}\int_0^{+\infty} d\rho'\int_0^{2(N+1) \pi}  \rho'\zeta(\rho' e^{{\rm i}\theta'}) d\theta' \\ &=\frac{1}{N+1}\int_0^{+\infty} d\rho'\int_0^{2\pi}  \rho'\zeta(\rho' e^{{\rm i}\theta'}) d\theta'
\\& =\frac{1}{N+1}\intr \zeta(y) dy.\end{aligned}$$

\end{proof}

\begin{lemma}\label{copy0}
Let $\ell,\,N\in\N$ with  $\frac{\ell}{N+1}\not\in \N$  and let $\zeta:\R^2\to\R$ be a continuous  $\frac{2\pi}{N+1}$-symmetric function, i.e. $\zeta(y e^{{\rm i}\frac{2\pi}{N+1}})=\zeta(y)$, such that $\zeta(y)|y|^{\ell}\in L^1(\R^2)$. Then 
$$ \intr  
\zeta(y) \Re (y^{\ell})dy=\intr   \zeta(y)
\Im (y^{\ell})dy=0.$$ As a corollary, if $\zeta(y)=\zeta(|y|)$ is radially symmetric, then $\intr \zeta(|y|)\Re (y^\ell)dy=\intr \zeta(|y|)\Im (y^\ell)dy=0$ for all $\ell\in\N$. 
\end{lemma}
\begin{proof}
Let $\omega = e^{\frac{2\pi i}{N+1}}$. By the invariance of $\zeta$ under the rotation $y \mapsto \omega^k y$, we have
$$
\int_{\mathbb{R}^2} \zeta(y)\, y^\ell \, dy
= \int_{\mathbb{R}^2} \zeta(\omega^k y)\, (\omega^k y)^\ell \, dy
= \omega^{k\ell} \int_{\mathbb{R}^2} \zeta(y)\, y^\ell \, dy.
$$
Averaging over $k=0,\dots,N$ gives
$$
\int_{\mathbb{R}^2} \zeta(y)\, y^\ell \, dy
= \frac{1}{N+1} \sum_{k=0}^N \omega^{k\ell}
\int_{\mathbb{R}^2} \zeta(y)\, y^\ell \, dy.
$$
Since
$$
\sum_{k=0}^N \omega^{k\ell} = 0 \quad \text{if } \frac{\ell}{N+1} \notin \mathbb{N},
$$
it follows that
$$
\int_{\mathbb{R}^2} \zeta(y)\, y^\ell \, dy = 0.
$$
Taking real and imaginary parts yields the claim.
\end{proof}

\end{document}